\documentclass[11pt]{article}

\usepackage[margin=1in]{geometry}
\usepackage[T1]{fontenc}
\usepackage{times}
\usepackage{microtype}

\usepackage{amsmath,amssymb,amsfonts}
\usepackage{amsthm}
\usepackage{mathtools}
\usepackage{url}
\usepackage{booktabs}
\usepackage{graphicx}
\usepackage{tikz}
\usepackage{algorithm}
\usepackage{algpseudocode}
\usepackage{placeins}   
\usepackage{xcolor}
\usepackage[numbers,sort&compress]{natbib}
\usepackage[hidelinks]{hyperref}

\makeatletter
\algrenewcommand\ALG@beginalgorithmic{\normalfont\fontsize{9}{11.5}\selectfont}
\makeatother

\graphicspath{{./}{images/}}

\providecommand{\doi}{}
\renewcommand{\doi}[1]{\url{https://doi.org/#1}}
\theoremstyle{plain}
\newtheorem{theorem}{Theorem}
\newtheorem{lemma}[theorem]{Lemma}
\newtheorem{proposition}[theorem]{Proposition}
\newtheorem{corollary}[theorem]{Corollary}

\theoremstyle{definition}
\newtheorem{example}[theorem]{Example}

\theoremstyle{remark}
\newtheorem{remark}[theorem]{Remark}

\newcommand{\R}{\mathbb{R}}
\newcommand{\C}{\mathbb{C}}
\newcommand{\Om}{\Omega}
\newcommand{\Oms}{\Omega_s}
\newcommand{\Omd}{\Omega_d}
\newcommand{\bu}{\mathbf{u}}

\newcommand{\bn}{\mathbf{n}}
\newcommand{\bt}{\boldsymbol{\tau}}
\newcommand{\K}{\mathcal{K}}
\newcommand{\M}{\mathcal{M}}
\newcommand{\Q}{\mathcal{Q}}
\newcommand{\Qh}{\widehat{\mathcal{Q}}}
\newcommand{\Sc}{\mathcal{S}}
\newcommand{\Sch}{\widehat{\mathcal{S}}}
\newcommand{\Herm}{\mathcal{H}}
\newcommand{\tr}{^{T}}
\DeclareMathOperator{\spec}{\sigma}
\DeclareMathOperator{\Real}{Re}
\DeclareMathOperator{\rank}{rank}
\DeclareMathOperator{\diag}{diag}
\DeclareMathOperator{\bidiag}{bidiag}

\begin{document}
\title{An efficient coupled block preconditioner for the Marker-and-Cell
       discretization of the Stokes--Darcy equations}

\author{Zhen Chao\thanks{Department of Mathematics, Western Washington
University, Bellingham, WA 98225, USA
(\texttt{chaoz2@wwu.edu}).  Corresponding author.}
\and
Tianwen Li\thanks{School of Mathematical Sciences, Tongji University,
Shanghai 201804, China.}}

\date{\today}

\maketitle

\begin{abstract}
We consider the iterative solution of the large, sparse
and nonsymmetric double saddle-point systems arising from the stationary
Stokes--Darcy equations in two dimensions, discretized by the
Marker-and-Cell (MAC) finite difference method.  Existing block triangular
preconditioners rest on approximations of two Schur complements, the second
of which is hard to approximate uniformly in the physical parameters.  We
propose instead a coupled preconditioner that keeps the interface coupling
and regularizes the constraint block by a small multiple of the identity, so
that the second Schur complement never has to be approximated.  We analyze
the spectrum of the preconditioned matrix, both in the ideal case and for the
inexact version used in computations, and derive from it a rule for scaling
the regularization parameter with the physical parameters.  Numerical results
show mesh-independent convergence of preconditioned GMRES and robustness with
respect to the viscosity and the permeability.
\end{abstract}



\section{Introduction}
\label{sec:intro}

The Stokes--Darcy equations describe fluid that flows freely in one part of a
domain and through a porous medium in another, the two regions being joined
along an interface.  Applications are numerous: groundwater flow and
contaminant transport, filtration devices, flow through vessel walls, and a
range of industrial processes in which a free fluid meets a porous solid
\cite{DiscacciatiQuarteroni2009}. Finite element \cite{LaytonSchieweckYotov2002,MardalTaiWinther2002}, finite
volume \cite{Luo2017} and finite difference
\cite{ShiueOngLai2018,RuiSun2020,SunRui2019} discretizations have all been
studied for the coupled problem.  We work here with the Marker-and-Cell (MAC)
scheme of Harlow and Welch \cite{HarlowWelch1965}, whose convergence theory
for the Stokes problem goes back to Nicolaides \cite{Nicolaides1992}.  The
scheme stores the two velocity components and the pressure at different
locations of a staggered grid.  It is inf-sup stable, it is easy to assemble,
and the sparsity pattern it produces turns out to be unusually convenient for
preconditioning, as we shall see in Section~\ref{sec:disc}.  Shiue, Ong and
Lai \cite{ShiueOngLai2018} proved stability and first-order convergence of a
MAC scheme for the Stokes--Darcy system, and Rui and Sun \cite{RuiSun2020}
extended the analysis to non-uniform grids.  We follow the discretization of
\cite{ShiueOngLai2018}.

The Darcy pressure, which we call the potential, the Stokes velocity and the
Stokes pressure are coupled through the interface conditions, and the
algebraic system inherits a $3\times3$
\emph{double saddle-point} structure,
\begin{equation}\label{eq:doubleSP}
  \begin{bmatrix} A_d & G\tr & 0\\ G & -A_s & B\tr\\ 0 & B & 0\end{bmatrix}
  \begin{bmatrix}\phi_h\\ -\bu_h\\ p_h\end{bmatrix}
  =\begin{bmatrix}g_1\\ g_2\\ -g_3\end{bmatrix},
\end{equation}
which is nonsymmetric, indefinite, and badly conditioned both as $h\to0$ and
as the viscosity $\nu$ and the permeability $\kappa$ separate.  Matrices of
this shape have received a good deal of attention lately.  Eigenvalue bounds
and invertibility criteria were obtained in
\cite{AliBeikBenzi2018,BradleyGreif2023,BeikGreifTrummer2024}; block diagonal
and block triangular preconditioners were analysed in
\cite{SognZulehner2019,BeiglSognZulehner2020,PearsonPotschka2024,HuangMa2019};
shift-splitting variants were proposed in \cite{Cao2019}.  Preconditioners for
the Stokes--Darcy problem in particular were studied in
\cite{CaiMuXu2009,Chidyagwai2016,BeikBenzi2022,HolterKuchtaMardal2021}, and
multigrid smoothers in \cite{Luo2017}.

Our starting point is the paper of Greif and He \cite{GreifHe2023}.  They
analysed three ideal block preconditioners $\M_1,\M_2,\M_3$ for
\eqref{eq:doubleSP} and, more importantly for the practitioner, derived
computable approximations of the two Schur complements
\begin{equation}\label{eq:schurs}
  \Sc_1 = A_s + GA_d^{-1}G\tr,\qquad \Sc_2 = B\Sc_1^{-1}B\tr .
\end{equation}
Their conclusion is that the interface matrix $G$ must appear in the
preconditioner if one wants scalability, and that among the three the block
triangular choice
\begin{equation*}
\widehat{\M}_3 = \begin{bmatrix}A_d&0&0\\ G&-\Sch_1&0\\
0&B&\Sch_2\end{bmatrix}
\end{equation*}
performs best.  Its performance depends,
however, on \emph{two} independent approximations: $\Sch_1\approx\Sc_1$, which
comes from an incomplete Cholesky factorization of $A_d$, and
$\Sch_2\approx\Sc_2$, which is obtained by a chain of sparsification
arguments.  $\Sc_2$ is a dense matrix
whose spectrum moves with $\nu$ and $\kappa$. A rather different construction was proposed by Bakrani Balani, Hajarian and
Bergamaschi \cite{BakraniBalani2023} for double saddle-point matrices with a
\emph{zero} $(2,2)$ block; see also \cite{BalaniBakrani2024}.  In their
preconditioner the second Schur complement is not approximated but replaced by
a scalar multiple of the identity, and the coupling block is retained in the
$(1,2)$ position.  Their setting does not cover \eqref{eq:doubleSP}, where the
$(2,2)$ block $-A_s$ is nonzero and nonsymmetric.

This paper adapts that idea to the MAC Stokes--Darcy system and analyses the
result.  The preconditioner we study is
\begin{equation*}
  \Q_\alpha=\begin{bmatrix}A_d & G\tr & 0\\ 0 & \Sc_1 & -B\tr\\
    0 & B & \alpha I\end{bmatrix},\qquad \alpha>0 ,
\end{equation*}
and applying $\Q_\alpha^{-1}$ costs one solve with $A_d$ and one solve with
the augmented operator $\Sc_1+\alpha^{-1}B\tr B$.  No approximation of $\Sc_2$
is needed anywhere. We show in
Theorems~\ref{thm:blockform} and \ref{thm:Q} that $\K\Q_\alpha^{-1}$ is block
lower triangular, and that its third diagonal block satisfies
$BR^{-1}B\tr=\Sc_2(\Sc_2+\alpha I)^{-1}$ with $R=\alpha\Sc_1+B\tr B$.  This
identity says that the regularization acts on the spectrum as a M\"obius map
applied to
$\spec(\Sc_2)$, so that $\spec(\Q_\alpha^{-1}\K)$ is known exactly: $1$ with
multiplicity $n^2$, $-1$ with multiplicity $2n^2-n$, and the $n^2$ numbers
$\theta_i/(\theta_i+\alpha)$.  Two consequences follow.  The third group
collapses onto $1$ at the rate $\mathcal{O}(\alpha)$
(Theorem~\ref{thm:cluster}), and, when the symmetric part of $\Sc_1$ is
positive definite, all of these eigenvalues lie in the disk
$|z-\tfrac12|<\tfrac12$ (Proposition~\ref{prop:disk}).
For the inexact preconditioner actually used in computations we prove that
every eigenvalue $\lambda\notin\{0,1\}$ is a root of a scalar quadratic whose
coefficients are Rayleigh quotients (Theorem~\ref{thm:inexact}).  This yields
a bound of order $\|E\|^{1/2}$ on the distance of the eigenvalues from the set
that contains the exact spectrum, with $E=\Sc_1-\Sch_1$ of rank at most $n$
and supported on the interface unknowns alone (Corollary~\ref{cor:pert}).  A by-product of the analysis is a rule for
choosing $\alpha$: the quantity that matters is $\alpha/\theta_{\min}$, and
since $\theta_{\min}$ scales like $\kappa$ for small permeability, the choice
$\alpha=\alpha_0\min\{1,\kappa\}$, with a fixed $\alpha_0$ for which we take
$10^{-2}$, keeps the clustering tight across
the whole parameter range we tested.

The paper is organized as follows.  Section~\ref{sec:model} recalls the model
and the interface conditions, and Section~\ref{sec:disc} the MAC
discretization and the structure of the blocks in \eqref{eq:doubleSP}.  The
preconditioner is introduced in Section~\ref{sec:precond} and analysed in
Section~\ref{sec:spectral}.  Section~\ref{sec:practical} deals with the
implementation and with the choice of $\alpha$.  Numerical results are in
Section~\ref{sec:numerics}, and Section~\ref{sec:conclusions} closes with a
summary.

\section{The Stokes--Darcy problem}
\label{sec:model}

Let $\Om\subset\R^2$ be a bounded domain split into two non-overlapping
subdomains, $\Om=\Oms\cup\Gamma\cup\Omd$.  The fluid flows freely in $\Oms$
and through a porous medium in $\Omd$, and
$\Gamma=\partial\Oms\cap\partial\Omd$ is the interface between them; see
Figure~\ref{fig:domain}.

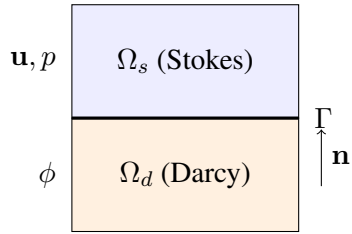
\begin{figure}[htbp]
\centering
\begin{tikzpicture}[scale=1.5]
  \draw[fill=blue!7] (0,1) rectangle (2,2);
  \draw[fill=orange!12] (0,0) rectangle (2,1);
  \draw[very thick] (0,1) -- (2,1);
  \node at (1,1.5) {$\Oms$ (Stokes)};
  \node at (1,0.5) {$\Omd$ (Darcy)};
  \node[right] at (2.05,1) {$\Gamma$};
  \node[left] at (-0.05,1.5) {$\bu,p$};
  \node[left] at (-0.05,0.5) {$\phi$};
  \draw[->] (2.2,0.4) -- (2.2,0.9); \node[right] at (2.2,0.65) {$\bn$};
\end{tikzpicture}
\caption{Two-dimensional domain for the Stokes--Darcy problem.  The unit
normal $\bn$ on the interface $\Gamma$ points from $\Omd$ into $\Oms$.}
\label{fig:domain}
\end{figure}

In the porous region, Darcy's law $\kappa^{-1}\bu^d+\nabla p^d=0$ together
with the mass balance $\nabla\cdot\bu^d=f^d$ gives, after eliminating the
Darcy velocity and writing $\phi=p^d$ for the Darcy pressure, which we
refer to as the Darcy potential,
\begin{equation}\label{eq:darcy}
  -\kappa\Delta\phi=f^d\quad\text{in }\Omd .
\end{equation}
The permeability $\kappa>0$ is taken constant, that is, the porous medium is
assumed homogeneous and isotropic.  In the free-flow region we have the Stokes
equations
\begin{equation}\label{eq:stokes}
  -\nu\Delta\bu+\nabla p=\mathbf{f}^s,\qquad \nabla\cdot\bu=0
  \quad\text{in }\Oms ,
\end{equation}
with velocity $\bu=(u,v)$, pressure $p$ and kinematic viscosity $\nu>0$.

Three conditions couple the two subproblems on $\Gamma$.  Taking $\Gamma$
horizontal, so that $\bn=(0,1)\tr$ and $\bt=(1,0)\tr$, they read
\begin{subequations}\label{eq:interface}
\begin{align}
  v &= -\kappa\,\frac{\partial\phi}{\partial y}, \label{eq:if-mass}\\
  p-\phi &= 2\nu\,\frac{\partial v}{\partial y}, \label{eq:if-normal}\\
  u &= \frac{\nu}{\alpha_{\mathrm{BJ}}}
       \left(\frac{\partial u}{\partial y}+\frac{\partial v}{\partial x}\right).
       \label{eq:if-bjs}
\end{align}
\end{subequations}
The first is conservation of mass, and forces the normal velocity to be
continuous.  The second balances normal forces and lets the pressure jump
across $\Gamma$.  The third is the Beavers--Joseph--Saffman slip condition
\cite{BeaversJoseph1967,Saffman1971}, in which $\alpha_{\mathrm{BJ}}>0$ is the
slip coefficient.  We write $\alpha_{\mathrm{BJ}}$ for that coefficient
throughout and reserve the plain symbol $\alpha$ for the regularization
parameter of Section~\ref{sec:precond}.  Dirichlet conditions
$\bu=\mathbf{g}^s_D$ on $\partial\Oms\setminus\Gamma$ and $\phi=g^d_D$ on
$\partial\Omd\setminus\Gamma$ complete the problem; they make the discrete
Darcy operator nonsingular, which is the case treated in \cite{GreifHe2023}
and here.

The difficulty of the model, both analytically and computationally, is that
the two subdomains are governed by operators of completely different type, and
that their scales separate as soon as $\nu$ and $\kappa$ do.  See
\cite{HolterKuchtaMardal2021,BeikBenzi2022} for parameter-robust formulations
that address this at the level of the continuous problem.

\section{The MAC discretization and the algebraic system}
\label{sec:disc}

Take $\Oms$ and $\Omd$ to be unit squares stacked across a horizontal
interface, and use a uniform staggered grid of mesh size $h=1/n$ in each of
them.  With $x_i=ih$ and $y_j=y_\Gamma+jh$, the unknowns sit as follows:
\begin{itemize}
\item $u_{i,j+\frac12}$ at $(x_i,y_{j+\frac12})$, $1\le i\le n-1$,
      $0\le j\le n-1$ \quad ($n^2-n$ unknowns);
\item $v_{i+\frac12,j}$ at $(x_{i+\frac12},y_j)$, $0\le i\le n-1$,
      $0\le j\le n-1$ \quad ($n^2$ unknowns, of which the $n$ values with
      $j=0$ lie exactly \emph{on} $\Gamma$);
\item $p_{i+\frac12,j+\frac12}$ at $(x_{i+\frac12},y_{j+\frac12})$,
      $0\le i,j\le n-1$ \quad ($n^2$ unknowns);
\item $\phi_{i+\frac12,j+\frac12}$ at $(x_{i+\frac12},y_{j+\frac12})$,
      $0\le i\le n-1$, $-n\le j\le -1$ \quad ($n^2$ unknowns).
\end{itemize}
Away from $\Gamma$ the equations are the usual five-point discretizations of
\eqref{eq:darcy} and \eqref{eq:stokes}.  On $\Gamma$ we follow
\cite{ShiueOngLai2018} and discretize \eqref{eq:interface} as
\begin{subequations}\label{eq:disc-interface}
\begin{align}
  v_{i+\frac12,0} &= -\kappa\,
     \frac{\phi_{i+\frac12,\frac12}-\phi_{i+\frac12,-\frac12}}{h},
     \label{eq:disc-mass}\\
  p_{i+\frac12,\frac12}-\phi_{i+\frac12,-\frac12}
     &= 2\nu\,\frac{v_{i+\frac12,1}-v_{i+\frac12,0}}{h},
     \label{eq:disc-normal}\\
  \frac{u_{i,\frac12}+u_{i,-\frac12}}{2}
     &= \frac{\nu}{\alpha_{\mathrm{BJ}}}
        \left(\frac{u_{i,\frac12}-u_{i,-\frac12}}{h}
             +\frac{v_{i+\frac12,0}-v_{i-\frac12,0}}{h}\right),
     \label{eq:disc-bjs}
\end{align}
\end{subequations}
where $\phi_{i+\frac12,\frac12}$ and $u_{i,-\frac12}$ are ghost values.  The
Darcy stencil in the row of cells next to $\Gamma$ contains
$\phi_{i+\frac12,\frac12}$; eliminating it with \eqref{eq:disc-mass} turns the
diagonal entry $4\kappa/h^2$ into $3\kappa/h^2$ and leaves behind the coupling
term $h^{-1}v_{i+\frac12,0}$.  The equations for the interface unknowns
$v_{i+\frac12,0}$ are \eqref{eq:disc-normal}.  The Stokes stencil in the row
of $u$ unknowns next to $\Gamma$ contains $u_{i,-\frac12}$, which is
eliminated with \eqref{eq:disc-bjs}.  Ghost values at the remaining Dirichlet
walls are removed by linear extrapolation.

Ordering the velocity unknowns as
$\bu_h=[u_h;\,v_h^\Gamma;\,v_h^{\mathrm{int}}]$ and symmetrizing the
off-diagonal blocks gives \eqref{eq:doubleSP} with coefficient matrix
\begin{equation}\label{eq:K}
  \K=\begin{bmatrix} A_d & G\tr & 0\\ G & -A_s & B\tr\\ 0 & B & 0\end{bmatrix}
  \in\R^{N\times N},\qquad N=4n^2-n .
\end{equation}
We summarize the structure of the blocks, following \cite[Sec.~3.4]{GreifHe2023}.

{\bf The Darcy block}
$A_d\in\R^{n^2\times n^2}$ is symmetric positive definite and partitions as
\begin{equation}\label{eq:Ad}
  A_d=\begin{bmatrix}A_{d,11} & A_{d,12}\\ A_{d,21} & A_{d,22}\end{bmatrix},
  \qquad A_{d,12}=A_{d,21}\tr,\qquad
  A_{d,21}=-\frac{\kappa}{h^2}\begin{bmatrix}0 & I_n\end{bmatrix},
\end{equation}
where $A_{d,22}\in\R^{n\times n}$ belongs to the $n$ cells adjacent to
$\Gamma$.

{\bf The interface coupling block}
$G\in\R^{(2n^2-n)\times n^2}$ is very sparse,
\begin{equation}\label{eq:G}
  G\tr=\begin{bmatrix}0 & 0 & 0\\ 0 & -I_n/h & 0\end{bmatrix},
\end{equation}
its single nonzero block linking the interface velocities $v_h^\Gamma$ to the
interface potentials.  In particular $\rank(G)=n$.

{\bf The Stokes block}
$A_s\in\R^{(2n^2-n)\times(2n^2-n)}$ is nonsymmetric,
\begin{equation}\label{eq:As}
  A_s=\begin{bmatrix}A_{11} & A_{12} & 0\\ 0 & A_{22} & A_{23}\\
      0 & A_{32} & A_{33}\end{bmatrix},\qquad
  A_{22}=\frac{2\nu}{h^2}I_n,\quad
  A_{23}=\begin{bmatrix}-A_{22} & 0\end{bmatrix},\quad
  A_{32}=\tfrac12 A_{23}\tr ,
\end{equation}
with $A_{11}=A_{11}\tr$ and $A_{33}=A_{33}\tr$ the positive definite discrete
Laplacians of the two velocity components.  The block $A_{12}$ is an
$(n-1)\times n$ upper bidiagonal matrix $c\cdot\bidiag[1,-1]$ stacked on a
zero block, where
\begin{equation}\label{eq:cbjs}
  c=\frac{2\nu^2}{h^2(2\nu+h\alpha_{\mathrm{BJ}})} ;
\end{equation}
it comes from eliminating $u_{i,-\frac12}$ with \eqref{eq:disc-bjs}.  All
eigenvalues of $A_s$ are positive \cite[Lem.~3.1]{GreifHe2023}.  Note that
this is weaker than positive definiteness of the symmetric part, a point we
return to in Remark~\ref{rem:assumption}.

{\bf The divergence block}
$B=\begin{bmatrix}B_x & B_0 & B_y\end{bmatrix}\in\R^{n^2\times(2n^2-n)}$ is
the negated discrete divergence, with
$B_0=\begin{bmatrix}I_n/h\\ 0\end{bmatrix}$, and it has full row rank $n^2$.

\begin{remark}\label{rem:blocks}
Everything that follows rests on two features of this structure.  The first is
that $G$ has rank $n=\mathcal{O}(\sqrt N)$ and lives on the interface
unknowns, so that $GA_d^{-1}G\tr$ is nonzero only in an $n\times n$ block.
The second is that $\Sc_2=B\Sc_1^{-1}B\tr$ involves the inverse of the
sparse but non-trivially structured $\Sc_1$, so $\Sc_2$ is a dense
$n^2\times n^2$ matrix that one never forms.  The preconditioner of
Section~\ref{sec:precond} exploits the first feature and sidesteps the second.
\end{remark}

\section{Block preconditioning}
\label{sec:precond}

\subsection{Ideal block preconditioners}

If $\Sc_1$ in \eqref{eq:schurs} is nonsingular, then $\K$ factors as
\begin{equation}\label{eq:ldu}
  \K=\underbrace{\begin{bmatrix} I&0&0\\ GA_d^{-1}&I&0\\
       0&-B\Sc_1^{-1}&I\end{bmatrix}}_{L}
     \underbrace{\begin{bmatrix} A_d&0&0\\ 0&-\Sc_1&0\\
       0&0&\Sc_2\end{bmatrix}}_{D}
     \underbrace{\begin{bmatrix} I&A_d^{-1}G\tr&0\\ 0&I&-\Sc_1^{-1}B\tr\\
       0&0&I\end{bmatrix}}_{U},
\end{equation}
and the three ideal preconditioners of \cite{GreifHe2023} are designed from
this factorization:
\begin{equation}\label{eq:M123}
\begin{gathered}
  \M_1=\begin{bmatrix}A_d&0&0\\ 0&\Sc_1&0\\ 0&0&\Sc_2\end{bmatrix},\qquad
  \M_2=\begin{bmatrix}A_d&0&0\\ G&\Sc_1&0\\ 0&0&\Sc_2\end{bmatrix},\\[6pt]
  \M_3=\begin{bmatrix}A_d&0&0\\ G&-\Sc_1&0\\ 0&B&\Sc_2\end{bmatrix}=LD .
\end{gathered}
\end{equation}
Their spectra are known largely in closed form.

\begin{theorem}[{\cite[Thm.~4.1]{GreifHe2023}}]\label{thm:M1}
The matrix $\M_1^{-1}\K$ has the eigenvalue $1$ with algebraic multiplicity at
least $n^2-n$, the eigenvalue $-1$ with algebraic multiplicity at least
$(n-1)^2$, and the eigenvalues $(-1\pm\sqrt5)/2$ with algebraic multiplicity at
least $n^2-n$ each.  At most $n$ eigenvalues are larger than $1$ and at most
$n$ lie in $(0,1)\setminus\{(-1+\sqrt5)/2\}$.
\end{theorem}


\begin{theorem}[{\cite[Thm.~4.3]{GreifHe2023}}]\label{thm:M2}
The matrix $\M_2^{-1}\K$ has exactly four distinct eigenvalues: $1$ with
multiplicity $n^2$, $-1$ with multiplicity $n^2-n$, and $(-1\pm\sqrt5)/2$ with
multiplicity $n^2$ each.
\end{theorem}

\begin{theorem}[{\cite[Thm.~4.4]{GreifHe2023}}]\label{thm:M3}
All eigenvalues of $\M_3^{-1}\K$ equal $1$, and the minimal polynomial of
$\M_3^{-1}\K$ is $(z-1)^3$; consequently GMRES converges in at most three iterations.
\end{theorem}

Theorem~\ref{thm:M3} is of course an idealization, since $\M_3$ needs exact
solves with $\Sc_1$ and with $\Sc_2$.  In practice $\Sc_1$ is replaced by an
approximation $\Sch_1$ built from an incomplete Cholesky factorization of
$A_d$, and $\Sc_2$ by a diagonal surrogate $\Sch_2$.  The resulting
$\widehat{\M}_3$ is the best of the three in practice, but it loses ground as
$\kappa$ decreases \cite[Tab.~8]{GreifHe2023}.

\subsection{A coupled preconditioner with scalar regularization}

All three preconditioners in \eqref{eq:M123} put an approximation of $\Sc_2$
in the $(3,3)$ position.  Following the idea of \cite{BakraniBalani2023} for
matrices with a zero $(2,2)$ block, we do something else: keep the exact
coupling block $G\tr$ in the $(1,2)$ position, keep the exact blocks $B$ and
$\pm B\tr$ in the third block row and column, and regularize the $(3,3)$
position with a multiple of the identity.  This gives
\begin{equation}\label{eq:Q}
  Q_\alpha=\begin{bmatrix}A_d & G\tr & 0\\
    0 & \Sc_1 & -B\tr\\ 0 & B & \alpha I_{n^2}\end{bmatrix},
    \qquad \alpha>0 .
\end{equation}
Comparing \eqref{eq:Q} with \eqref{eq:K}, the preconditioner shares the blocks
$A_d$, $G\tr$ and $B$ with $\K$ itself, and $B\tr$ up to sign.  It differs in
four places
only: the $(2,1)$ block is dropped, the $(2,2)$ block is $\Sc_1$ rather than
$-A_s$, the sign of the $(2,3)$ block is reversed, and $\alpha I$ replaces the
zero $(3,3)$ block.  Nothing in the construction involves $\Sc_2$.
Before turning to the spectrum we record two elementary facts used
below.  Here $\Herm(M)=\tfrac12(M+M\tr)$ denotes the symmetric
part of a real matrix $M$, so that $\Real(y^*My)=y^*\Herm(M)y$ for every
$y\in\C^m$.

\begin{lemma}\label{lem:invert}
Let $A_d$ and $\Sc_1$ be nonsingular and let $\alpha>0$.  Then $\Q_\alpha$ is
nonsingular if and only if $-\alpha\notin\spec(\Sc_2)$.  If $\Herm(\Sc_1)$ is
positive definite, then $\Q_\alpha$ is nonsingular for every $\alpha>0$.
\end{lemma}

\begin{proof}
The $(2,1)$ and $(3,1)$ blocks of $\Q_\alpha$ vanish and $A_d$ is nonsingular,
so $\Q_\alpha$ is nonsingular precisely when
\[C=\left[\begin{smallmatrix}\Sc_1&-B\tr\\ B&\alpha I\end{smallmatrix}\right]\]
is.  Eliminating the $(1,1)$ block of $C$ gives
\[\det C=\det(\Sc_1)\det(\alpha I+B\Sc_1^{-1}B\tr)
       =\det(\Sc_1)\det(\alpha I+\Sc_2),\]
which vanishes exactly when $-\alpha\in\spec(\Sc_2)$. Suppose now $\Herm(\Sc_1)\succ0$ and $Cy=0$ with $y=(y_1,y_2)$.  Then
$\alpha y_2=-By_1$ and $\Sc_1y_1=B\tr y_2$, so
$y_1^*\Sc_1y_1=y_1^*B\tr y_2=-\alpha^{-1}\|By_1\|^2$.  Taking real parts,
$y_1^*\Herm(\Sc_1)y_1+\alpha^{-1}\|By_1\|^2=0$, whence $y_1=0$ and then
$y_2=0$.
\end{proof}

\begin{lemma}\label{lem:rhp}
If $\Herm(\Sc_1)$ is positive definite, then every eigenvalue $\theta$ of
$\Sc_2=B\Sc_1^{-1}B\tr$ has $\Real\theta>0$.
\end{lemma}

\begin{proof}
Let $\Sc_2y=\theta y$ with $\|y\|=1$ and put $w=\Sc_1^{-1}B\tr y$.  Since $B$
has full row rank, $B\tr y\neq0$, so $w\neq0$; moreover $\Sc_1w=B\tr y$ and
$Bw=\theta y$.  Hence $w^*\Sc_1w=w^*B\tr y=(Bw)^*y=\overline{\theta}$, and
taking real parts gives $\Real\theta=w^*\Herm(\Sc_1)w>0$.
\end{proof}

\begin{remark}\label{rem:assumption}
The hypothesis $\Herm(\Sc_1)\succ0$ is sufficient but not necessary, and for
the MAC matrices it is not always satisfied.  It holds at all but one of the parameter
pairs of Table~\ref{tab:assumption}, the exception being $\kappa=\nu=1$,
where it fails on the grids with $n\ge32$.  What fails there is the
definiteness of $\Herm(A_s)$ itself, which turns mildly indefinite once the
mesh resolves the interface layer; for $\kappa<1$ the positive definite
interface block $T=\mathcal{O}(\kappa^{-1})$ of \eqref{eq:GAG} below, which
enters $\Sc_1=A_s+GA_d^{-1}G\tr$, more than compensates at every other parameter
pair, and $\Herm(\Sc_1)\succ0$ is restored.  Section~\ref{sec:numerics} quantifies both
statements.  What is perhaps surprising is that the conclusion of
Lemma~\ref{lem:rhp} is satisfied in all of our experiments even where its
hypothesis is not.  In any case the clustering result of
Theorem~\ref{thm:cluster} does not use the hypothesis at all.
\end{remark}

\section{Spectral analysis}
\label{sec:spectral}

Throughout this section $A_d$ and $\Sc_1$ are nonsingular, $\alpha>0$, and
$-\alpha\notin\spec(\Sc_2)$, so that $\Q_\alpha$ is nonsingular by
Lemma~\ref{lem:invert}.  The matrices $\Q_\alpha^{-1}\K$ and
$\K\Q_\alpha^{-1}$ are similar, and it is convenient to work with the latter.

\subsection{The exact preconditioner}

\begin{theorem}\label{thm:blockform}
Let $R=\alpha\Sc_1+B\tr B$.  Then $R$ is nonsingular and
\begin{equation}\label{eq:KQinv}
  \K\Q_\alpha^{-1}=\begin{bmatrix} I & 0 & 0\\
    GA_d^{-1} & -I & 0\\ 0 & \alpha BR^{-1} & BR^{-1}B\tr\end{bmatrix}.
\end{equation}
Consequently
\begin{equation}\label{eq:specsplit}
  \spec(\Q_\alpha^{-1}\K)=\{1\}^{n^2}\cup\{-1\}^{2n^2-n}
    \cup\spec\!\left(BR^{-1}B\tr\right),
\end{equation}
where the exponents denote algebraic multiplicities.
\end{theorem}

\begin{proof}
Write $\Q_\alpha=\left[\begin{smallmatrix}A_d & \Gamma_{12}\\ 0 &
C\end{smallmatrix}\right]$ with
$\Gamma_{12}=\begin{bmatrix}G\tr&0\end{bmatrix}$ and
$C=\left[\begin{smallmatrix}\Sc_1&-B\tr\\ B&\alpha I\end{smallmatrix}\right]$,
so that
\begin{equation*}
\Q_\alpha^{-1}=\begin{bmatrix}A_d^{-1} &
  -A_d^{-1}\Gamma_{12}C^{-1}\\ 0 & C^{-1}\end{bmatrix}.
\end{equation*}
The first block row of $\K$ is $\begin{bmatrix}A_d&\Gamma_{12}\end{bmatrix}$,
and
\[
  \begin{bmatrix}A_d&\Gamma_{12}\end{bmatrix}\Q_\alpha^{-1}
  =\begin{bmatrix}I & -\Gamma_{12}C^{-1}+\Gamma_{12}C^{-1}\end{bmatrix}
  =\begin{bmatrix}I&0&0\end{bmatrix}.
\]
The second block row of $\K$ is
$\begin{bmatrix}G & -A_s & B\tr\end{bmatrix}$.  Its first component is
$GA_d^{-1}$, and its last two components are
\[
  \begin{aligned}
  -GA_d^{-1}\begin{bmatrix}G\tr&0\end{bmatrix}C^{-1}
  +\begin{bmatrix}-A_s & B\tr\end{bmatrix}C^{-1}
  &=\begin{bmatrix}-(A_s+GA_d^{-1}G\tr) & B\tr\end{bmatrix}C^{-1}
  =-\begin{bmatrix}\Sc_1 & -B\tr\end{bmatrix}C^{-1},
  \end{aligned}
\]
which equals $-\begin{bmatrix}I&0\end{bmatrix}$, because
$\begin{bmatrix}\Sc_1&-B\tr\end{bmatrix}$ is the first block row of $C$.
The third block row of $\K$ is $\begin{bmatrix}0&B&0\end{bmatrix}$, so it
remains to compute
$\begin{bmatrix}B&0\end{bmatrix}C^{-1}=:\begin{bmatrix}X&Y\end{bmatrix}$.  This
means $X\Sc_1+YB=B$ and $-XB\tr+\alpha Y=0$.  The second equation gives
$Y=\alpha^{-1}XB\tr$, and substituting it into the first gives
$X(\Sc_1+\alpha^{-1}B\tr B)=B$, that is, $X=\alpha BR^{-1}$ and
$Y=BR^{-1}B\tr$.  Nonsingularity of $R$ follows from that of $C$: eliminating
the $(2,2)$ block of $C$ gives
\[\det C=\det(\alpha I_{n^2})\det(\Sc_1+\alpha^{-1}B\tr B)
      =\alpha^{n^2}\alpha^{-(2n^2-n)}\det R=\alpha^{\,n-n^2}\det R.\]  This
proves \eqref{eq:KQinv}, and \eqref{eq:specsplit} follows because the spectrum
of a block triangular matrix is the union of the spectra of its diagonal
blocks.
\end{proof}

The next result puts the third group of eigenvalues  $\spec\!\left(BR^{-1}B\tr\right)$ in closed form.  It is the
observation on which the rest of the paper rests: the regularization $\alpha
I$ acts on the spectrum of $\Q_\alpha^{-1}\K$ as a M\"obius transformation
applied to the spectrum of the second Schur complement.

\begin{theorem}\label{thm:Q}
With $R=\alpha\Sc_1+B\tr B$ we have
\begin{equation}\label{eq:mobius}
  BR^{-1}B\tr=\Sc_2\,(\Sc_2+\alpha I)^{-1}.
\end{equation}
The $n^2$ remaining eigenvalues in \eqref{eq:specsplit} are therefore
\begin{equation}\label{eq:mus}
  \mu_i=\frac{\theta_i}{\theta_i+\alpha},\qquad i=1,\dots,n^2,
\end{equation}
where $\theta_1,\dots,\theta_{n^2}$ are the eigenvalues of
$\Sc_2=B\Sc_1^{-1}B\tr$, counted with algebraic multiplicity.  None of the
$\mu_i$ equals $1$, so the eigenvalue $1$ has multiplicity exactly $n^2$; none
equals $0$ when $\Sc_2$ is nonsingular; and
\[
  \spec(\Q_\alpha^{-1}\K)=\{1\}^{n^2}\cup\{-1\}^{2n^2-n}\cup
  \Big\{\tfrac{\theta_i}{\theta_i+\alpha}\Big\}_{i=1}^{n^2}.
\]
\end{theorem}

\begin{proof}
By the Sherman--Morrison--Woodbury formula,
\[
  R^{-1}=(\alpha\Sc_1+B\tr B)^{-1}
   =\frac1\alpha\Sc_1^{-1}
    -\frac1{\alpha^2}\Sc_1^{-1}B\tr
      \Big(I+\tfrac1\alpha B\Sc_1^{-1}B\tr\Big)^{-1}B\Sc_1^{-1}.
\]
Multiply by $B$ on the left and $B\tr$ on the right, and write
$\Sc_2=B\Sc_1^{-1}B\tr$:
\[
  \begin{aligned}
  BR^{-1}B\tr&=\frac1\alpha\Sc_2-\frac1{\alpha^2}
     \Sc_2\Big(I+\tfrac1\alpha\Sc_2\Big)^{-1}\Sc_2\\
   &=\frac1\alpha\Sc_2\Big(I+\tfrac1\alpha\Sc_2\Big)^{-1}
     \Big[\Big(I+\tfrac1\alpha\Sc_2\Big)-\tfrac1\alpha\Sc_2\Big]\\
   &=\Sc_2(\alpha I+\Sc_2)^{-1},
  \end{aligned}
\]
which is \eqref{eq:mobius}.  Thus $BR^{-1}B\tr=r(\Sc_2)$ for the rational
function $r(z)=z/(z+\alpha)$, which is analytic on a neighbourhood of
$\spec(\Sc_2)$ because $-\alpha\notin\spec(\Sc_2)$.  The spectral mapping
theorem gives $\spec(BR^{-1}B\tr)=r(\spec(\Sc_2))$, and since $r$ is
injective on that neighbourhood, the algebraic multiplicity of each
$r(\theta_i)$ equals that of $\theta_i$.
Finally $r(\theta)=1$ has no finite solution, and $r(\theta)=0$ only for
$\theta=0$.
\end{proof}

\begin{remark}
It is worth comparing Theorem~\ref{thm:Q} with Theorems~\ref{thm:M1} and
\ref{thm:M2}.  For $\M_1$ and $\M_2$ the eigenvalues that are not $\pm1$ are
the fixed algebraic numbers $(-1\pm\sqrt5)/2$.  These do not depend on the
problem at all, they have moduli $1.618$ and $0.618$, and no choice of
parameters moves them.  For $\Q_\alpha$ the corresponding group is a M\"obius
image of $\spec(\Sc_2)$, and decreasing $\alpha$ drives it towards $1$.  That
is the structural reason for the behaviour reported in
Section~\ref{sec:numerics}.
\end{remark}

\subsection{Localization and clustering}

\begin{theorem}\label{thm:cluster}
Assume that $\Sc_2$ is nonsingular, and set
$\theta_{\min}=\min_i|\theta_i|>0$.  If $0<\alpha\le\theta_{\min}/2$,
then every eigenvalue $\mu_i$ in \eqref{eq:mus} satisfies
\begin{equation}\label{eq:clusterbound}
  |1-\mu_i|=\frac{\alpha}{|\theta_i+\alpha|}\le\frac{2\alpha}{\theta_{\min}} .
\end{equation}
The third group of eigenvalues therefore collapses onto $1$ at the rate
$\mathcal{O}(\alpha)$, and for small $\alpha$ the spectrum of
$\Q_\alpha^{-1}\K$ consists of two clusters, one at $-1$ of size $2n^2-n$ and
one at $1$ of size $2n^2$.  
\end{theorem}

\begin{proof}
By \eqref{eq:mus}, $1-\mu_i=\alpha/(\theta_i+\alpha)$.  If
$\alpha\le\theta_{\min}/2$ then
$|\theta_i+\alpha|\ge|\theta_i|-\alpha\ge\theta_{\min}/2$.
\end{proof}

Under the sufficient condition of Lemma~\ref{lem:rhp} one can say considerably
more, and for every $\alpha>0$ rather than only for small $\alpha$.

\begin{proposition}\label{prop:disk}
If $\Herm(\Sc_1)\succ0$, then for every $\alpha>0$ all the eigenvalues
$\mu_i$ of \eqref{eq:mus} lie in the open disk
$\mathbb{D}=\{z\in\C:|z-\tfrac12|<\tfrac12\}$, so that in particular
$0<\Real\mu_i<1$ and $|\mu_i|<1$.  Moreover
\begin{equation}\label{eq:sharpbound}
  |1-\mu_i|\le\frac{\alpha}{\sqrt{\theta_{\min}^2+\alpha^2}}<1 .
\end{equation}
\end{proposition}

\begin{proof}
By Lemma~\ref{lem:rhp}, $\Real\theta_i>0$ for every $i$; in particular no
$\theta_i$ vanishes, so $\theta_{\min}=\min_i|\theta_i|>0$.  Fix $i$ and
write $\theta=\theta_i$ and $\mu=\mu_i=\theta/(\theta+\alpha)$.

For the first claim,
\[
  \Big|\mu-\tfrac12\Big|^2-\tfrac14
  =\frac{|\theta-\alpha|^2-|\theta+\alpha|^2}{4|\theta+\alpha|^2}
  =\frac{-4\alpha\Real\theta}{4|\theta+\alpha|^2}<0 ,
\]
so $\mu$ lies in $\mathbb{D}$, and the two consequences
$0<\Real\mu<1$ and $|\mu|<1$ follow because $\mathbb{D}$ is contained in
the strip $0<\Real z<1$ and in the unit disk.

For \eqref{eq:sharpbound}, subtracting from $1$ gives
\begin{equation}\label{eq:onemmu}
  1-\mu=\frac{(\theta+\alpha)-\theta}{\theta+\alpha}
        =\frac{\alpha}{\theta+\alpha},
  \qquad\text{so}\qquad
  |1-\mu|=\frac{\alpha}{|\theta+\alpha|},
\end{equation}
the last step because $\alpha>0$ is real.  Expanding the denominator,
\[
  |\theta+\alpha|^2=(\theta+\alpha)\overline{(\theta+\alpha)}
  =|\theta|^2+\alpha\big(\theta+\overline{\theta}\big)+\alpha^2
  =|\theta|^2+2\alpha\Real\theta+\alpha^2 .
\]
The cross term $2\alpha\Real\theta$ is positive, since $\alpha>0$ and
$\Real\theta>0$, and $|\theta|\ge\theta_{\min}$ by the definition of
$\theta_{\min}$; hence
\[
  |\theta+\alpha|^2\ \ge\ |\theta|^2+\alpha^2\ \ge\
  \theta_{\min}^2+\alpha^2 .
\]
Taking square roots and inserting this in \eqref{eq:onemmu} gives
$|1-\mu|\le\alpha/\sqrt{\theta_{\min}^2+\alpha^2}$, which is $<1$
because $\theta_{\min}>0$.  Since $i$ was arbitrary, \eqref{eq:sharpbound}
holds for every eigenvalue in \eqref{eq:mus}.
\end{proof}

\begin{remark}
The bound \eqref{eq:sharpbound} is sharper than the clustering bound
\eqref{eq:clusterbound} and needs no restriction on \[\alpha:
\alpha/\sqrt{\theta_{\min}^2+\alpha^2}\le\alpha/\theta_{\min}
\le2\alpha/\theta_{\min}\] for every $\alpha>0$.  It is also asymptotically
exact, since the inequalities above are equalities in the limit
$\Real\theta\downarrow0$, $|\theta|\to\theta_{\min}$; Table~\ref{tab:bound}
measures the gap.
\end{remark}

\begin{remark}\label{rem:gmres}
The two clusters at $\pm1$ already account for $3n^2-n$ of the $4n^2-n$
eigenvalues, and the remaining $n^2$ join the cluster at $1$ as
$\alpha\downarrow0$.  The matrix $\K\Q_\alpha^{-1}$ is not normal, so its eigenvalues alone do not
determine the GMRES convergence curve \cite{BenziGolubLiesen2005}, and we make
no rigorous statement about the convergence rate.  What we can say is that in
the experiments of Section~\ref{sec:numerics} the two-cluster structure is
accompanied by iteration counts of around nine, independently of the mesh.
\end{remark}

\subsection{The inexact preconditioner}
\label{sec:inexact}

In computations $\Sc_1=A_s+GA_d^{-1}G\tr$ is replaced by an approximation
$\Sch_1$, which gives
\begin{equation}\label{eq:Qhat}
  \Qh_\alpha=\begin{bmatrix}A_d & G\tr & 0\\
    0 & \Sch_1 & -B\tr\\ 0 & B & \alpha I\end{bmatrix},
  \qquad E:=\Sc_1-\Sch_1 .
\end{equation}
By \eqref{eq:G} the term $GA_d^{-1}G\tr$ is nonzero only in the $n\times n$
diagonal block belonging to the interface velocities $v_h^\Gamma$, and the
approximation of Section~\ref{sec:practical} touches that block alone.  Hence:

\begin{lemma}\label{lem:ranknE}
$E$ is supported on the $n$ rows and columns corresponding to $v_h^\Gamma$.
In particular $\rank(E)\le n=\mathcal{O}(\sqrt N)$, and $\|E\|_2$ is the
norm of the perturbation made in that single $n\times n$ block.
\end{lemma}

The theorem below describes the whole spectrum of $\Qh_\alpha^{-1}\K$ by means
of a scalar quadratic equation whose coefficients are Rayleigh quotients.  The
exact preconditioner is the special case $E=0$, and the perturbation bound
that follows is an easy consequence.

\begin{theorem}\label{thm:inexact}
Let $\Qh_\alpha$ be nonsingular.  Then
\begin{enumerate}
\item[(i)] $\lambda=1$ is an eigenvalue of $\Qh_\alpha^{-1}\K$ with algebraic
  multiplicity at least $n^2$;
\item[(ii)] if $\Sc_2$ is nonsingular, $\lambda=0$ is not an eigenvalue;
\item[(iii)] for every eigenvalue $\lambda\notin\{0,1\}$ there is a vector
  $y\in\C^{2n^2-n}$ with $\|y\|=1$ such that
  \begin{equation}\label{eq:quadratic}
    (\hat s+\gamma)\,\lambda^2+(\hat s+\varepsilon)\,\lambda-\gamma=0,
    \qquad
    \hat s=y^*\Sch_1y,\quad \gamma=\frac{\|By\|^2}{\alpha},\quad
    \varepsilon=y^*Ey .
  \end{equation}
\end{enumerate}
If $E=0$, the roots of \eqref{eq:quadratic} are $\lambda=-1$ and
$\lambda=\gamma/(\hat s+\gamma)$, in agreement with
Theorems~\ref{thm:blockform} and \ref{thm:Q}.
\end{theorem}

\begin{proof}
As at the start of Section~\ref{sec:spectral}, $\Qh_\alpha^{-1}\K$ and
$\K\Qh_\alpha^{-1}$ are similar, so we may argue with the latter.  The first
block row of $\K$ is $\begin{bmatrix}A_d&\Gamma_{12}\end{bmatrix}$, and
$\Qh_\alpha$ reproduces it exactly, since the approximation $\Sch_1$ enters
only the second block row.  The computation in the proof of
Theorem~\ref{thm:blockform} therefore still gives
$\begin{bmatrix}A_d&\Gamma_{12}\end{bmatrix}\Qh_\alpha^{-1}
=\begin{bmatrix}I&0&0\end{bmatrix}$, so $\K\Qh_\alpha^{-1}$ is block lower
triangular with leading diagonal block $I_{n^2}$, which proves (i).

For (ii) and (iii) consider $\K w=\lambda\Qh_\alpha w$ with
$w=(w_1,w_2,w_3)\neq0$.  The three block rows are
\begin{subequations}\label{eq:eig}
\begin{align}
  (1-\lambda)\,(A_dw_1+G\tr w_2) &= 0,\label{eq:eig1}\\
  Gw_1-A_sw_2+B\tr w_3 &= \lambda(\Sch_1w_2-B\tr w_3),\label{eq:eig2}\\
  Bw_2 &= \lambda(Bw_2+\alpha w_3).\label{eq:eig3}
\end{align}
\end{subequations}
Suppose $\lambda=0$.  Then \eqref{eq:eig3} gives $Bw_2=0$, and since
$1-\lambda\neq0$, \eqref{eq:eig1} gives $w_1=-A_d^{-1}G\tr w_2$.  Substituting
both into \eqref{eq:eig2} yields $\Sc_1w_2=B\tr w_3$, so that
$w_2=\Sc_1^{-1}B\tr w_3$ and
\[
  0=Bw_2=B\Sc_1^{-1}B\tr w_3=\Sc_2w_3 .
\]
If $\Sc_2$ is nonsingular this forces $w_3=0$; then $\Sc_1w_2=0$ gives
$w_2=0$, and finally $w_1=0$, contradicting $w\neq0$.  This proves (ii).
(Equivalently, $\lambda=0$ is an eigenvalue of $\Qh_\alpha^{-1}\K$ exactly
when $\K$ is singular, and \eqref{eq:ldu} gives
\[\det\K=\det A_d\,\det(-\Sc_1)\,\det\Sc_2,\] so the hypothesis on $\Sc_2$
cannot be dropped.)

Let now $\lambda\notin\{0,1\}$.  From \eqref{eq:eig1} we have
$w_1=-A_d^{-1}G\tr w_2$, so the left-hand side of \eqref{eq:eig2} is
\[-(A_s+GA_d^{-1}G\tr)w_2+B\tr w_3=-\Sc_1w_2+B\tr w_3,\] and \eqref{eq:eig2}
becomes
\begin{equation}\label{eq:eig2b}
  (\lambda\Sch_1+\Sc_1)w_2=(1+\lambda)B\tr w_3 .
\end{equation}
Equation \eqref{eq:eig3} gives $w_3=\alpha^{-1}\lambda^{-1}(1-\lambda)Bw_2$.
Note that $w_2\neq0$, since $w_2=0$ would force $w_3=0$ and then $w_1=0$; we
may therefore normalize $\|w_2\|=1$.  Substituting $w_3$ into
\eqref{eq:eig2b} and writing $\Sc_1=\Sch_1+E$,
\[
  \lambda(1+\lambda)\Sch_1w_2+\lambda Ew_2
    =\frac{(1-\lambda)(1+\lambda)}{\alpha}B\tr Bw_2 .
\]
Multiplying on the left by $w_2^*$ and setting $y=w_2$,
$\hat s=y^*\Sch_1y$, $\varepsilon=y^*Ey$ and
$\gamma=\alpha^{-1}y^*B\tr By=\alpha^{-1}\|By\|^2$ gives
$\lambda(1+\lambda)\hat s+\lambda\varepsilon=(1-\lambda^2)\gamma$, which is
\eqref{eq:quadratic}.  When $E=0$ we have $\varepsilon=0$ and
\eqref{eq:quadratic} factors as
$(1+\lambda)\big[(\hat s+\gamma)\lambda-\gamma\big]=0$.
\end{proof}

\begin{corollary}\label{cor:pert}
Let $\lambda\notin\{0,1\}$ be an eigenvalue of $\Qh_\alpha^{-1}\K$, let
$y$, $\hat s$, $\gamma$, $\varepsilon$ be as in Theorem~\ref{thm:inexact},
assume $\hat s+\gamma\neq0$, and put $\mu_0=\gamma/(\hat s+\gamma)$.  Then
\begin{equation}\label{eq:pertbound}
  \min\big\{|\lambda+1|,\;|\lambda-\mu_0|\big\}
  \;\le\;\left(\frac{|\lambda|\,\|E\|_2}{|\hat s+\gamma|}\right)^{1/2}.
\end{equation}
If $\Herm(\Sch_1)\succ0$, the assumption $\hat s+\gamma\neq0$ is automatic,
$|\hat s+\gamma|\ge\Real\hat s\ge\lambda_{\min}(\Herm(\Sch_1))$, and $\mu_0$
lies in the closed disk \[\overline{\mathbb{D}}=\{z:|z-\tfrac12|\le\tfrac12\},\]
so every eigenvalue of $\Qh_\alpha^{-1}\K$ other than $1$ lies within a
distance $\big(\|E\|_2\,|\lambda|/\lambda_{\min}(\Herm(\Sch_1))\big)^{1/2}$
of the set $\{-1\}\cup\overline{\mathbb{D}}$.
\end{corollary}

\begin{proof}
The monic quadratic with roots $-1$ and $\mu_0$ is
$\lambda^2+(1-\mu_0)\lambda-\mu_0$, and
$(\hat s+\gamma)(1-\mu_0)=\hat s$, $(\hat s+\gamma)\mu_0=\gamma$; hence
\eqref{eq:quadratic} can be rewritten as
\[(\hat s+\gamma)(\lambda+1)(\lambda-\mu_0)=-\lambda\varepsilon.\]  Taking
absolute values and using $|\varepsilon|\le\|E\|_2$ gives
\[|\lambda+1|\,|\lambda-\mu_0|\le|\lambda|\|E\|_2/|\hat s+\gamma|,\] and the
smaller of two nonnegative numbers never exceeds the square root of their
product.  For the last statement,
\[|\hat s+\gamma|\ge\Real(\hat s+\gamma)=y^*\Herm(\Sch_1)y+\gamma
 \ge\lambda_{\min}(\Herm(\Sch_1))\] because $\gamma\ge0$.  Writing
$\hat s=a+\mathrm{i}b$ with $a=y^*\Herm(\Sch_1)y>0$ gives
\[|\mu_0-\tfrac12|^2-\tfrac14
 =\big(|\gamma-\hat s|^2-|\gamma+\hat s|^2\big)/\big(4|\gamma+\hat s|^2\big)
 =-\gamma a/|\gamma+\hat s|^2\le0,\]
so $\mu_0\in\overline{\mathbb{D}}$.  Note that $\mu_0$ need not be real, since
$\Sch_1$ is nonsymmetric and $\hat s$ is a complex Rayleigh quotient.
\end{proof}

In words, Corollary~\ref{cor:pert} says that the quality of the preconditioner
is controlled by one small matrix, the $n\times n$ interface block: every
eigenvalue of the inexact preconditioned matrix lies within
$\mathcal{O}\big((|\lambda|\,\|E\|)^{1/2}\big)$ of the set
$\{-1\}\cup\overline{\mathbb{D}}$ that contains the exact spectrum.  The
bound is stated relative to that set rather than to the individual exact
eigenvalues, since $\mu_0$ is built from the perturbed eigenvector and need
not itself be an eigenvalue of $\Q_\alpha^{-1}\K$.  The experiments in
Section~\ref{sec:numerics} verify this.

\section{Practical implementation}
\label{sec:practical}

\subsection{Approximation of the first Schur complement}

The sparsity of $G$ in \eqref{eq:G} gives
\begin{equation}\label{eq:GAG}
  GA_d^{-1}G\tr=\begin{bmatrix}0&0&0\\ 0&T&0\\ 0&0&0\end{bmatrix},
  \qquad T=\frac{1}{h^2}\big(A_d^{-1}\big)_{\Gamma\Gamma}\in\R^{n\times n},
\end{equation}
the subscript $\Gamma\Gamma$ denoting the trailing $n\times n$ block.
Following \cite{GreifHe2023}, let $A_d=FF\tr$ be the Cholesky factorization
partitioned as
$F=\left[\begin{smallmatrix}F_{11}&0\\ F_{21}&F_{22}\end{smallmatrix}\right]$
with $F_{22}\in\R^{n\times n}$; then $T=h^{-2}F_{22}^{-T}F_{22}^{-1}$.  The
identity requires the factor in the natural ordering with the interface
unknowns last, in which $F$ fills the entire band of $A_d$: $\mathcal{O}(n^3)$
storage and $\mathcal{O}(n^4)=\mathcal{O}(N^2)$ work, far more than a
fill-reducing sparse factorization of $A_d$ costs.  (Computing $T$ column by
column from such a factorization is possible, and is how the ``exact'' rows of
Table~\ref{tab:inexact} were obtained, but it takes $n$ solves with $A_d$ per
setup.)  Following \cite{GreifHe2023} we therefore compute an incomplete
Cholesky factorization $A_d\approx\widetilde F\widetilde F\tr$ in that
ordering, with drop tolerance $\varepsilon_{\rm IC}$, and take
\begin{equation}\label{eq:S1hat}
  \widehat T=\frac{1}{h^2}\widetilde F_{22}^{-T}\widetilde F_{22}^{-1},
  \qquad
  \Sch_1=\begin{bmatrix}A_{11}&A_{12}&0\\ 0&A_{22}+\widehat T&A_{23}\\
     0&A_{32}&A_{33}\end{bmatrix} .
\end{equation}
Only an $n\times n$ triangular factor is inverted here, at negligible cost.
Lemma~\ref{lem:ranknE} applies with $E=\diag(0,T-\widehat T,0)$, so that
$\|E\|_2=\|T-\widehat T\|_2$.

\subsection{Applying the preconditioner}

Partition $r=(r_1,r_2,r_3)$ conformally with \eqref{eq:Qhat} and solve
$\Qh_\alpha s=r$ by block elimination.  It is convenient to start from the
third block row, whose diagonal block $\alpha I$ is trivially invertible:
$s_3=\alpha^{-1}(r_3-Bs_2)$, and substituting this into the second block row
gives
\begin{equation}\label{eq:Wsolve}
  \underbrace{\Big(\Sch_1+\tfrac1\alpha B\tr B\Big)}_{=:W_\alpha}s_2
  = r_2+\tfrac1\alpha B\tr r_3 ,
\end{equation}
after which $A_ds_1=r_1-G\tr s_2$.  Algorithm~\ref{alg:apply} collects the
steps.

\begin{algorithm}[htbp]
\caption{Application of $\Qh_\alpha^{-1}$ to a vector $r=(r_1,r_2,r_3)$}
\label{alg:apply}
\begin{algorithmic}[1]
\State \textbf{Setup (once):} incomplete Cholesky factor $\widetilde F$ of
       $A_d$ and $\widehat T$ from \eqref{eq:S1hat}; sparse factorizations of
       $A_d$ and of $W_\alpha=\Sch_1+\alpha^{-1}B\tr B$.
\State Solve $W_\alpha s_2 = r_2+\alpha^{-1}B\tr r_3$.
\State $s_3\gets\alpha^{-1}(r_3-Bs_2)$.
\State Solve $A_d s_1 = r_1-G\tr s_2$.
\State \Return $s=(s_1,s_2,s_3)$.
\end{algorithmic}
\end{algorithm}

Steps 2--4 involve $A_d$, $B$, $G$ and $\Sch_1$ and nothing else.  The second
Schur complement, whose approximation is the delicate ingredient of
$\widehat\M_1$, $\widehat\M_2$ and $\widehat\M_3$, has disappeared; this is
the practical side of Theorem~\ref{thm:Q}, where its effect was absorbed into
the scalar $\alpha$. There is a familiar object hiding in \eqref{eq:Wsolve}.  The operator
$W_\alpha=\Sch_1+\alpha^{-1}B\tr B$ is the augmented Lagrangian velocity
operator with penalty parameter $\alpha^{-1}$ \cite{BenziOlshanskii2006}, so
$\Q_\alpha$ can be read as an augmented Lagrangian treatment of the
velocity--pressure saddle point embedded in an exact treatment of the
Darcy--interface coupling.  The usual caveat of augmented Lagrangian methods
applies here too: $\operatorname{cond}(W_\alpha)$ grows like $\alpha^{-1}$,
which puts a floor under the useful values of $\alpha$ in finite precision.
We return to this in Section~\ref{sec:alphachoice}.

As for cost, each GMRES iteration requires two solves with the stored
factorizations of $A_d$ and $W_\alpha$ plus a few sparse matrix--vector
products, so the asymptotic cost per iteration is the same as for
$\widehat\M_3$.  The setup is the more expensive of the two on coarser
grids, since $W_\alpha$ is denser than $\Sch_1$;
Table~\ref{tab:timing} quantifies this.  When many systems share the
same matrix, as in time stepping, parameter continuation or uncertainty
quantification, the setup is paid once and the smaller iteration count
translates directly into a shorter run time.

\subsection{Choice of the regularization parameter}
\label{sec:alphachoice}

By Theorem~\ref{thm:cluster} the quantity that governs the clustering is not
$\alpha$ but the ratio $\alpha/\theta_{\min}$, where
$\theta_{\min}=\min_i|\theta_i|$ and $\theta_i\in\spec(\Sc_2)$.  Values of
$\theta_{\min}$ over a wide range of parameters are collected in
Table~\ref{tab:assumption}.  They show, first, that $\theta_{\min}$ is
essentially independent of $h$, which is one reason why the iteration counts
of Section~\ref{sec:numerics} do not grow with the mesh.  They also show how
it depends on the physical parameters.  For small permeability
$\theta_{\min}$ is proportional to $\kappa$ and, to the accuracy of the
table, independent of $\nu$.  This is what one expects
from the interface block $T$ in \eqref{eq:GAG}, which scales like
$\kappa^{-1}$ and dominates $\Sc_1$ in the interface directions.  For
$\kappa=1$ the proportionality saturates at a level that grows as $\nu$
decreases, about $0.15$--$0.17$ at $\nu=1$ and about $3$ at $\nu=10^{-2}$,
which is the
regime where the bulk Stokes part of $\Sc_1$, of size $\mathcal{O}(\nu)$,
determines the smallest eigenvalue instead.  In short,
$\theta_{\min}\approx\min\{c_1(\nu),\,c_2\kappa\}$ with $c_2\approx8$.
A fixed $\alpha$ must therefore lose its effect as $\kappa\to0$, because
$\alpha/\theta_{\min}\to\infty$.

This suggests scaling $\alpha$ with the permeability,
\begin{equation}\label{eq:alphascaled}
  \alpha=\alpha_0\min\{1,\kappa\},\qquad \alpha_0=10^{-2},
\end{equation}
which keeps $\alpha/\theta_{\min}$ at about $\alpha_0/c_2\approx10^{-3}$
throughout the small-$\kappa$ regime, and at $\alpha_0/c_1(\nu)\le0.07$
elsewhere for the values of $\nu$ tested.  The value $\alpha_0=10^{-2}$ is
a compromise.  Smaller values cluster the spectrum more tightly by
Theorem~\ref{thm:cluster}, but the growth of
$\operatorname{cond}(W_\alpha)=\mathcal{O}(\alpha^{-1})$ eventually spoils the
accuracy of the inner solves.  The first effect can be seen in
Table~\ref{tab:alpha}; the second becomes visible in
Table~\ref{tab:robust-nu2} on the finest grid, where the scaled rule pushes
$\alpha$ down to $10^{-10}$ and the counts start to creep up.  The rule
\eqref{eq:alphascaled} is tested in Tables~\ref{tab:robust-nu1} and
\ref{tab:robust-nu2}.

\section{Numerical experiments}
\label{sec:numerics}
We use our own implementation of restarted GMRES
\cite{SaadSchultz1986}
with modified Gram--Schmidt orthogonalization and Givens rotations rather than
a built-in routine, so that the stopping test and the reported iteration
counts are unambiguous and exactly reproducible.  The method is applied to the
left-preconditioned system and stops when
\[
  \|\M^{-1}(b-\K x_k)\|_2\le \texttt{tol}\cdot\|\M^{-1}b\|_2 .
\]
Unless stated otherwise we take GMRES$(20)$, $\texttt{tol}=10^{-8}$, at most
$500$ iterations, a zero initial guess, drop tolerance
$\varepsilon_{\rm IC}=10^{-2}$ in the incomplete Cholesky factorization, and
$\alpha=10^{-2}$.  In the tables ``IT'' is the total number of iterations and
``$-$'' means that $500$ iterations were reached without convergence.  Systems
with $A_d$, $\Sch_1$ and $W_\alpha$ are solved with sparse $LU$ factorizations
computed once in the setup phase.
All results reported below were produced with MATLAB~2026b on a Mac Studio with an Apple M4 Max processor and 64~GB of memory.
Times are wall-clock seconds and exclude the assembly of
$\K$, which is the same for all methods.

\subsection{Test problems}

We use two manufactured solutions taken from
\cite{ShiueOngLai2018,GreifHe2023}.

\begin{example}\label{ex:1}
$\Oms=[0,1]\times[1,2]$, $\Omd=[0,1]^2$, $y_\Gamma=1$, and
\[
  u=-\tfrac1\pi e^y\sin(\pi x),\qquad v=(e^y-e)\cos(\pi x),
\]
\[
  p=2e^y\cos(\pi x),\qquad \phi=(e^y-ey)\cos(\pi x).
\]
The interface conditions \eqref{eq:interface} force
$\nu=\alpha_{\mathrm{BJ}}=1$ here, so this example serves to verify the
discretization rather than to vary parameters.
\end{example}

\begin{example}\label{ex:3}
$\Oms=[0,1]^2$, $\Omd=[0,1]\times[-1,0]$, $y_\Gamma=0$, and
\[
  u=\eta'(y)\cos x,\quad v=\eta(y)\sin x,\quad p=0,\quad \phi=e^y\sin x,
\]
\[
  \eta(y)=-\kappa-\frac{y}{2\nu}
        +\Big(\frac\kappa2-\frac{\alpha_{\mathrm{BJ}}}{4\nu^2}\Big)y^2 .
\]
The three coefficients of $\eta$ are fixed by \eqref{eq:interface}, and no
constraint is placed on $\nu$, $\kappa$ or $\alpha_{\mathrm{BJ}}$.  All
parameter studies below use this example, with $\alpha_{\mathrm{BJ}}=1$
throughout.  Its exact Stokes pressure vanishes identically, so the column
$\|p-p_h\|$ of Table~\ref{tab:converr-ex3} below measures the decay of the
spurious discrete pressure rather than an approximation error in the usual
sense.
\end{example}

Dirichlet data are taken from the exact solutions.  Table~\ref{tab:dimensions}
lists the problem sizes.

\begin{table}[htbp]
\centering
\caption{Number of unknowns of the discrete Stokes--Darcy system.}
\label{tab:dimensions}
\small
\setlength{\tabcolsep}{4pt}
\begin{tabular}{rrrrr}
\toprule
$n$ & $\dim\phi_h$ & $\dim\mathbf{u}_h$ & $\dim p_h$ & $N=4n^2-n$ \\
\midrule
16 & 256 & 496 & 256 & 1,008 \\
32 & 1,024 & 2,016 & 1,024 & 4,064 \\
64 & 4,096 & 8,128 & 4,096 & 16,320 \\
128 & 16,384 & 32,640 & 16,384 & 65,408 \\
256 & 65,536 & 130,816 & 65,536 & 261,888 \\
512 & 262,144 & 523,776 & 262,144 & 1,048,064 \\
1024 & 1,048,576 & 2,096,128 & 1,048,576 & 4,193,280 \\
\bottomrule
\end{tabular}
\end{table}

\subsection{Verification of the discretization}

Table~\ref{tab:converr-ex1} gives the discrete $L^2$ errors and the observed
convergence orders for Example~\ref{ex:1} with
$\nu=\kappa=\alpha_{\mathrm{BJ}}=1$.  The velocity components and the Darcy
potential converge with order two.  The pressure is slower, rising from
$1.714$ on the coarsest pair of grids to $1.851$ on the finest; the order is
still climbing at $n=512$.  Table~\ref{tab:converr-ex3} does the same for
Example~\ref{ex:3} with $\nu=1$ and $\kappa=10^{-2}$, where all four
components converge with order one.  Both outcomes are consistent with the
theory of \cite{ShiueOngLai2018}, which guarantees first order under a mesh
constraint.  We use the same grids as \cite{GreifHe2023} so that the two sets
of tables can be compared line by line.

\begin{table}[htbp]
\centering
\caption{Discrete $L^2$ errors and observed convergence orders for Example~1 with $\nu=\kappa=\alpha_{\rm BJ}=1$.}
\label{tab:converr-ex1}
\footnotesize
\setlength{\tabcolsep}{4pt}
\begin{tabular}{rcccccccc}
\toprule
& \multicolumn{2}{c}{$\|u-u_h\|$} & \multicolumn{2}{c}{$\|v-v_h\|$} & \multicolumn{2}{c}{$\|p-p_h\|$} & \multicolumn{2}{c}{$\|\phi-\phi_h\|$} \\
\cmidrule(lr){2-3}\cmidrule(lr){4-5}\cmidrule(lr){6-7}\cmidrule(lr){8-9}
$n$ & error & order & error & order & error & order & error & order \\
\midrule
32 & $6.88\!\times\!10^{-4}$ & --- & $9.09\!\times\!10^{-4}$ & --- & $2.37\!\times\!10^{-2}$ & --- & $2.41\!\times\!10^{-4}$ & --- \\
64 & $1.73\!\times\!10^{-4}$ & 1.989 & $2.29\!\times\!10^{-4}$ & 1.989 & $7.23\!\times\!10^{-3}$ & 1.714 & $6.05\!\times\!10^{-5}$ & 1.995 \\
128 & $4.35\!\times\!10^{-5}$ & 1.996 & $5.74\!\times\!10^{-5}$ & 1.996 & $2.11\!\times\!10^{-3}$ & 1.776 & $1.52\!\times\!10^{-5}$ & 1.998 \\
256 & $1.09\!\times\!10^{-5}$ & 1.998 & $1.44\!\times\!10^{-5}$ & 1.999 & $5.98\!\times\!10^{-4}$ & 1.820 & $3.79\!\times\!10^{-6}$ & 1.999 \\
512 & $2.72\!\times\!10^{-6}$ & 1.999 & $3.59\!\times\!10^{-6}$ & 2.000 & $1.66\!\times\!10^{-4}$ & 1.851 & $9.48\!\times\!10^{-7}$ & 2.000 \\
\bottomrule
\end{tabular}
\end{table}
\begin{table}[htbp]
\centering
\caption{Discrete $L^2$ errors and observed convergence orders for Example~2 with $\nu=1$, $\kappa=10^{-2}$ and $\alpha_{\rm BJ}=1$.}
\label{tab:converr-ex3}
\footnotesize
\setlength{\tabcolsep}{4pt}
\begin{tabular}{rcccccccc}
\toprule
& \multicolumn{2}{c}{$\|u-u_h\|$} & \multicolumn{2}{c}{$\|v-v_h\|$} & \multicolumn{2}{c}{$\|p-p_h\|$} & \multicolumn{2}{c}{$\|\phi-\phi_h\|$} \\
\cmidrule(lr){2-3}\cmidrule(lr){4-5}\cmidrule(lr){6-7}\cmidrule(lr){8-9}
$n$ & error & order & error & order & error & order & error & order \\
\midrule
32 & $9.20\!\times\!10^{-5}$ & --- & $1.14\!\times\!10^{-4}$ & --- & $1.38\!\times\!10^{-2}$ & --- & $1.03\!\times\!10^{-3}$ & --- \\
64 & $4.48\!\times\!10^{-5}$ & 1.039 & $5.35\!\times\!10^{-5}$ & 1.094 & $7.02\!\times\!10^{-3}$ & 0.975 & $4.86\!\times\!10^{-4}$ & 1.077 \\
128 & $2.22\!\times\!10^{-5}$ & 1.016 & $2.59\!\times\!10^{-5}$ & 1.046 & $3.54\!\times\!10^{-3}$ & 0.987 & $2.37\!\times\!10^{-4}$ & 1.035 \\
256 & $1.10\!\times\!10^{-5}$ & 1.006 & $1.28\!\times\!10^{-5}$ & 1.022 & $1.78\!\times\!10^{-3}$ & 0.994 & $1.17\!\times\!10^{-4}$ & 1.016 \\
512 & $5.50\!\times\!10^{-6}$ & 1.003 & $6.33\!\times\!10^{-6}$ & 1.011 & $8.91\!\times\!10^{-4}$ & 0.997 & $5.83\!\times\!10^{-5}$ & 1.008 \\
\bottomrule
\end{tabular}
\end{table}

\subsection{The unpreconditioned spectrum}
\label{sec:kspec}

Before turning to the preconditioners we check the assembly of $\K$ itself.
Figure~\ref{fig:kspectrum} shows its eigenvalues for $n=32$ at four pairs of
physical parameters, reproducing the experiment of
\cite[Fig.~6]{GreifHe2023}, and Table~\ref{tab:kspectrum} compares the extreme
real parts with the values quoted there.  Every digit given in
\cite{GreifHe2023} is reproduced but one.  We obtain $81.9$ and $-8183.0$ for
$\nu=1$, $\kappa=10^{-2}$; $8189.5$ for $\nu=10^{-2}$, $\kappa=1$; and $90.0$
and $-90.8$ for $\nu=10^{-4}$, $\kappa=10^{-8}$; for $\nu=\kappa=1$ only the
order of magnitude $\pm10^4$ is quoted there, which we reproduce as well.  So
are the qualitative features: the imaginary parts stay below $2.5$ in magnitude while the real parts
reach $\pm10^4$, and for $\nu=10^{-4}$, $\kappa=10^{-8}$ the spectrum is
entirely real and nearly symmetric about the origin.  


\begin{figure}[htbp]
\centering
\includegraphics[width=0.9\textwidth]{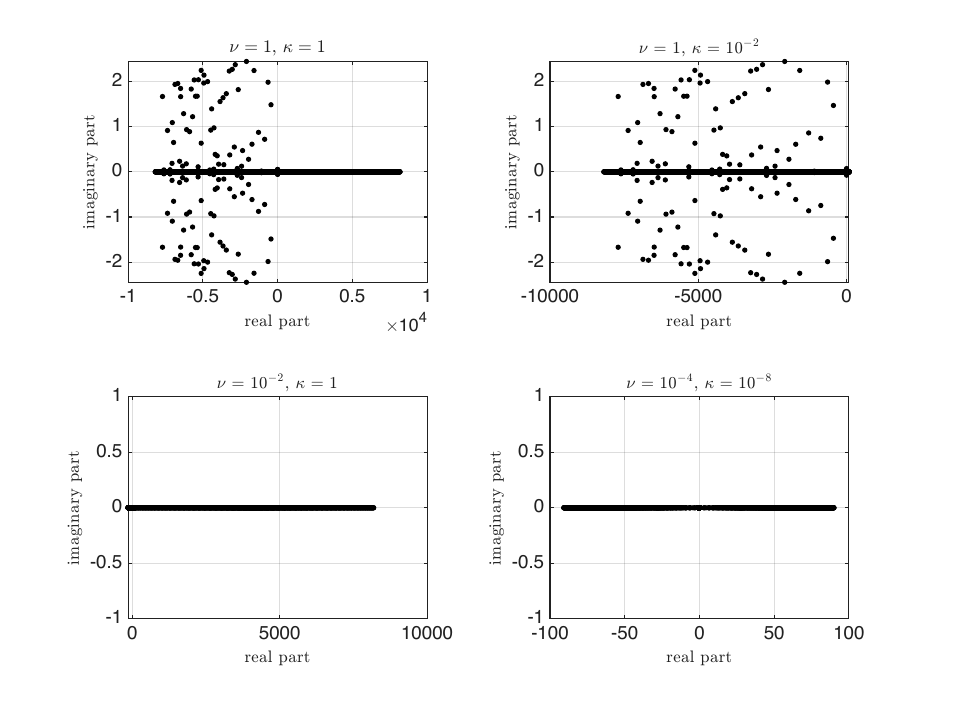}
\caption{Eigenvalues of the unpreconditioned matrix $\K$ for $n=32$ and four
pairs of physical parameters (Example~\ref{ex:3}), reproducing
\cite[Fig.~6]{GreifHe2023}.}
\label{fig:kspectrum}
\end{figure}

\begin{table}[htbp]
\centering
\caption{Extreme eigenvalues of the unpreconditioned matrix $\K$ for $n=32$ (Example~2), compared with the values quoted in \cite[Sec.~5.2]{GreifHe2023} for the same experiment.  The column \cite{GreifHe2023} lists their $\max_i\operatorname{Re}\lambda_i$ and $\min_i\operatorname{Re}\lambda_i$.}
\label{tab:kspectrum}
\small
\setlength{\tabcolsep}{4pt}
\begin{tabular}{ccrrcrr}
\toprule
$\nu$ & $\kappa$ & $\max_i\operatorname{Re}\lambda_i$ & $\min_i\operatorname{Re}\lambda_i$ & \cite{GreifHe2023} & $\max_i|\operatorname{Im}\lambda_i|$ & \# complex \\
\midrule
$1$ & $1$ & $8189.5$ & $-8183.0$ & $\approx\pm10^{4}$ & $2.45$ & 138 \\
$1$ & $10^{-2}$ & $81.9$ & $-8183.0$ & $81.9$ / $-8183.0$ & $2.45$ & 138 \\
$10^{-2}$ & $1$ & $8189.5$ & $-140.1$ & $8189.5$ / $-0.4$ & $0.00$ & 0 \\
$10^{-4}$ & $10^{-8}$ & $90.0$ & $-90.8$ & $90.0$ / $-90.8$ & $0.00$ & 0 \\
\bottomrule
\end{tabular}
\end{table}

\subsection{Verification of the spectral analysis}
\label{sec:verify}

Table~\ref{tab:idealspec} lists the distinct eigenvalues of the four
preconditioned matrices for $n=8$ with exact Schur complements, together with
their algebraic multiplicities.  For $\M_1$ the four multiplicities of
Theorem~\ref{thm:M1} are attained with equality, and not merely as the lower
bounds the proof of \cite[Thm.~4.1]{GreifHe2023} supplies: we count
$n^2-n=56$ for each of $1$ and $(-1\pm\sqrt5)/2$, and $(n-1)^2=49$ for $-1$.
The last of these does not appear as a single row of the table, because the
clustering tolerance separates a few nearby but distinct values from it; they
are absorbed in the ``$25$ further'' count.  Exactly $n=8$ eigenvalues exceed
$1$ and exactly $n=8$ lie in $(0,1)\setminus\{(-1+\sqrt5)/2\}$, so both bounds
of that theorem are attained as well, and the $2n-1=15$ eigenvalues it leaves
unlocated \cite[Rem.~4.2]{GreifHe2023} turn out to be small perturbations of
$-1$ and $(-1-\sqrt5)/2$.  This is consistent with the numerical confirmation
reported in \cite{GreifHe2023}, obtained here with an independent
implementation.  For $\M_2$ the computation reproduces
Theorem~\ref{thm:M2} exactly: four distinct eigenvalues $1$, $-1$ and
$(-1\pm\sqrt5)/2$, with multiplicities $64$, $56$, $64$, $64$, that is $n^2$,
$n^2-n$, $n^2$, $n^2$.  For $\M_3$ a single eigenvalue $1$ of multiplicity
$N=248$ is found, as Theorem~\ref{thm:M3} predicts.  For $\Q_\alpha$ we obtain
exactly $n^2=64$ eigenvalues equal to $1$ and $2n^2-n=120$ equal to $-1$, plus
$64$ further eigenvalues clustered near $1$.  We also checked directly that
these $64$ values agree with  $\theta_i/(\theta_i+\alpha)$
of $\spec(\Sc_2)$.

\begin{table}[htbp]
\centering
\caption{Distinct eigenvalues of the preconditioned matrices for the ideal preconditioners with $n=8$ ($\nu=\kappa=1$, $\alpha=10^{-2}$), computed with a dense eigensolver and clustered to a tolerance of $10^{-4}$. The counts agree with Theorems~\ref{thm:M1}, \ref{thm:M2}, \ref{thm:M3} and \ref{thm:Q}.}
\label{tab:idealspec}
\footnotesize
\setlength{\tabcolsep}{4pt}
\begin{tabular}{lcl}
\toprule
preconditioner & \# distinct & eigenvalues (algebraic multiplicity) \\
\midrule
$\mathcal{M}_1$ & 28 & $-1.6180$ (56), $0.6180$ (56), $1.0000$ (56), $\dots$ (25 further) \\
$\mathcal{M}_2$ & 4 & $-1.6180$ (64), $1.0000$ (64), $0.6180$ (64), $-1.0000$ (56) \\
$\mathcal{M}_3$ & 1 & $1.0000$ (248) \\
$\mathcal{Q}_\alpha$ & 25 & $-1.0000$ (120), $1.0000$ (64), $0.9901$ (36), $\dots$ (22 further) \\
\bottomrule
\end{tabular}
\end{table}

The entry for $\M_3$ deserves a word, because a preconditioned matrix with a
single eigenvalue can look like a programming error.  It is not one.  By
construction $\M_3=LD$ in the factorization \eqref{eq:ldu}, so
$\M_3^{-1}\K=(LD)^{-1}LDU=U$, the unit upper triangular factor of that
factorization, and every eigenvalue of a unit triangular matrix equals $1$.
What the table does not show, and what matters for the solver, is that $U$ is
far from the identity: it differs from $I$ in two blocks, it is not
diagonalizable, and by Theorem~\ref{thm:M3} its minimal polynomial is
$(z-1)^3$.  That is why GMRES needs three iterations rather than one.  Exact
block triangular preconditioners for saddle-point matrices behave this way
generally; see \cite{BenziGolubLiesen2005,ElmanSilvesterWathen2014}.  It is
also why $\M_3$ is of theoretical interest only, since forming it calls for
exact solves with both Schur complements.

Table~\ref{tab:assumption} collects the spectral quantities that enter
Section~\ref{sec:spectral}.  The symmetric part of $A_s$ is positive definite
on the two coarse grids but turns mildly indefinite for $n\ge32$, at every
value of $\nu$ and $\kappa$ ($A_s$ does not depend on $\kappa$), the negative
eigenvalue being at most about $0.1\%$ of $\|A_s\|_2$ in magnitude on these
grids.  The symmetric part
of $\Sc_1$ inherits this only for $\kappa=\nu=1$; for the smaller
permeabilities the interface block $T$ restores definiteness.  This is the
situation described in Remark~\ref{rem:assumption}.  Even where the hypothesis
fails, $\Real\theta_i>0$ holds in every case, and $\theta_{\min}$, which is
the only quantity Theorem~\ref{thm:cluster} needs, stays bounded away from
zero on the grids of the table.  Its dependence on the
physical parameters is strong, and it is this dependence that motivates the
scaling \eqref{eq:alphascaled}.

\begin{table}[htbp]
\centering
\caption{Verification of the hypotheses of Section~\ref{sec:spectral}. $\mathcal{H}(\cdot)$ denotes the symmetric part and $\theta_i$ are the eigenvalues of the second Schur complement $\mathcal{S}_2=B\mathcal{S}_1^{-1}B^{T}$. The quantity $\theta_{\min}$ is the only one entering the clustering bound of Theorem~\ref{thm:cluster}; it is bounded away from zero on all grids considered.}
\label{tab:assumption}
\small
\setlength{\tabcolsep}{4pt}
\begin{tabular}{cccccccc}
\toprule
$\nu$ & $\kappa$ & $n$ & $\lambda_{\min}(\mathcal{H}(A_s))$ & $\lambda_{\min}(\mathcal{H}(\mathcal{S}_1))$ & $\min_i\operatorname{Re}\theta_i$ & $\theta_{\min}$ & $\max_i|\theta_i|$ \\
\midrule
$1$ & $1$ & 8 & $8.00\!\times\!10^{0}$ & $8.43\!\times\!10^{0}$ & $1.72\!\times\!10^{-1}$ & $1.72\!\times\!10^{-1}$ & $1.65\!\times\!10^{0}$ \\
$1$ & $1$ & 16 & $5.41\!\times\!10^{0}$ & $6.00\!\times\!10^{0}$ & $1.60\!\times\!10^{-1}$ & $1.60\!\times\!10^{-1}$ & $1.68\!\times\!10^{0}$ \\
$1$ & $1$ & 32 & $-4.69\!\times\!10^{0}$ & $-3.68\!\times\!10^{0}$ & $1.55\!\times\!10^{-1}$ & $1.55\!\times\!10^{-1}$ & $1.69\!\times\!10^{0}$ \\
$1$ & $1$ & 48 & $-2.09\!\times\!10^{1}$ & $-1.94\!\times\!10^{1}$ & $1.54\!\times\!10^{-1}$ & $1.54\!\times\!10^{-1}$ & $1.70\!\times\!10^{0}$ \\
$1$ & $10^{-4}$ & 8 & $8.00\!\times\!10^{0}$ & $1.38\!\times\!10^{1}$ & $6.94\!\times\!10^{-4}$ & $6.94\!\times\!10^{-4}$ & $1.01\!\times\!10^{0}$ \\
$1$ & $10^{-4}$ & 16 & $5.41\!\times\!10^{0}$ & $1.39\!\times\!10^{1}$ & $7.30\!\times\!10^{-4}$ & $7.30\!\times\!10^{-4}$ & $1.11\!\times\!10^{0}$ \\
$1$ & $10^{-4}$ & 32 & $-4.69\!\times\!10^{0}$ & $1.40\!\times\!10^{1}$ & $7.01\!\times\!10^{-4}$ & $7.01\!\times\!10^{-4}$ & $1.31\!\times\!10^{0}$ \\
$1$ & $10^{-4}$ & 48 & $-2.09\!\times\!10^{1}$ & $1.40\!\times\!10^{1}$ & $6.66\!\times\!10^{-4}$ & $6.66\!\times\!10^{-4}$ & $1.45\!\times\!10^{0}$ \\
$10^{-2}$ & $1$ & 8 & $8.31\!\times\!10^{-2}$ & $1.79\!\times\!10^{-1}$ & $3.40\!\times\!10^{0}$ & $3.40\!\times\!10^{0}$ & $1.09\!\times\!10^{2}$ \\
$10^{-2}$ & $1$ & 16 & $5.89\!\times\!10^{-2}$ & $1.81\!\times\!10^{-1}$ & $2.96\!\times\!10^{0}$ & $2.96\!\times\!10^{0}$ & $1.19\!\times\!10^{2}$ \\
$10^{-2}$ & $1$ & 32 & $-3.89\!\times\!10^{-2}$ & $1.80\!\times\!10^{-1}$ & $2.76\!\times\!10^{0}$ & $2.76\!\times\!10^{0}$ & $1.29\!\times\!10^{2}$ \\
$10^{-2}$ & $1$ & 48 & $-1.99\!\times\!10^{-1}$ & $1.78\!\times\!10^{-1}$ & $2.70\!\times\!10^{0}$ & $2.70\!\times\!10^{0}$ & $1.34\!\times\!10^{2}$ \\
$10^{-2}$ & $10^{-4}$ & 8 & $8.31\!\times\!10^{-2}$ & $1.93\!\times\!10^{-1}$ & $7.15\!\times\!10^{-4}$ & $7.15\!\times\!10^{-4}$ & $1.00\!\times\!10^{2}$ \\
$10^{-2}$ & $10^{-4}$ & 16 & $5.89\!\times\!10^{-2}$ & $1.95\!\times\!10^{-1}$ & $8.02\!\times\!10^{-4}$ & $8.02\!\times\!10^{-4}$ & $1.00\!\times\!10^{2}$ \\
$10^{-2}$ & $10^{-4}$ & 32 & $-3.89\!\times\!10^{-2}$ & $1.95\!\times\!10^{-1}$ & $8.87\!\times\!10^{-4}$ & $8.87\!\times\!10^{-4}$ & $1.00\!\times\!10^{2}$ \\
$10^{-2}$ & $10^{-4}$ & 48 & $-1.99\!\times\!10^{-1}$ & $1.95\!\times\!10^{-1}$ & $9.34\!\times\!10^{-4}$ & $9.34\!\times\!10^{-4}$ & $1.00\!\times\!10^{2}$ \\
$10^{-4}$ & $10^{-8}$ & 8 & $8.32\!\times\!10^{-4}$ & $1.95\!\times\!10^{-3}$ & $7.15\!\times\!10^{-8}$ & $7.15\!\times\!10^{-8}$ & $1.00\!\times\!10^{4}$ \\
$10^{-4}$ & $10^{-8}$ & 16 & $5.90\!\times\!10^{-4}$ & $1.97\!\times\!10^{-3}$ & $8.03\!\times\!10^{-8}$ & $8.03\!\times\!10^{-8}$ & $1.00\!\times\!10^{4}$ \\
$10^{-4}$ & $10^{-8}$ & 32 & $-3.87\!\times\!10^{-4}$ & $1.97\!\times\!10^{-3}$ & $8.91\!\times\!10^{-8}$ & $8.91\!\times\!10^{-8}$ & $1.00\!\times\!10^{4}$ \\
$10^{-4}$ & $10^{-8}$ & 48 & $-1.98\!\times\!10^{-3}$ & $1.97\!\times\!10^{-3}$ & $9.42\!\times\!10^{-8}$ & $9.42\!\times\!10^{-8}$ & $1.00\!\times\!10^{4}$ \\
\bottomrule
\end{tabular}
\end{table}

Figure~\ref{fig:spectra} shows the spectra of the four preconditioned matrices
for $n=16$.  The points $(-1\pm\sqrt5)/2$ characteristic of $\M_1$ and $\M_2$
are plainly visible, and so is the collapse of the spectrum of $\M_3^{-1}\K$
onto the single point $1$.  For $\Q_\alpha^{-1}\K$ one sees the two tight
clusters at $\pm1$ of Theorem~\ref{thm:cluster}.  


\begin{figure}[htbp]
\centering
\includegraphics[width=0.95\textwidth]{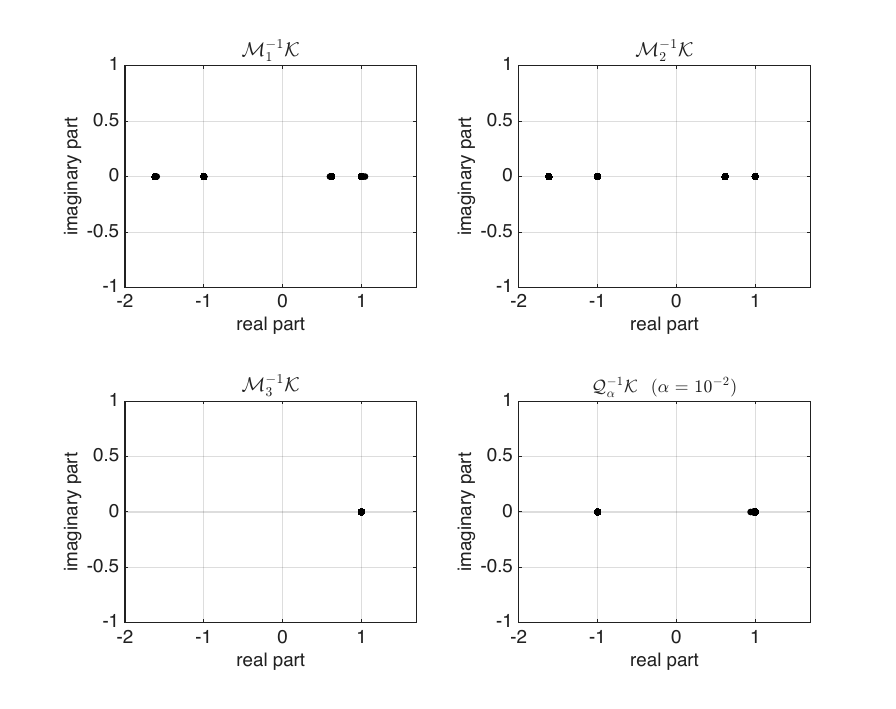}
\caption{Eigenvalues of the preconditioned matrices for Example~\ref{ex:3}
with $n=16$, $\nu=\kappa=1$, using the exact Schur complements
($\alpha=10^{-2}$ for $\Q_\alpha$).}
\label{fig:spectra}
\end{figure}

\subsection{The role of \texorpdfstring{$\alpha$}{alpha}}

Table~\ref{tab:alpha} reports iteration counts and CPU times for the five
values $10^{-3}$, $10^{-2}$, $10^{-1}$, $1$ and $10$ of the regularization
parameter.  The count falls monotonically with $\alpha$, from about $40$ at
$\alpha=10$ to $6$--$8$ at $\alpha=10^{-3}$, and for each fixed $\alpha$ it
grows only mildly with $n$: from $20$ to $27$ at $\alpha=1$, and by at most
two iterations over the whole range of grids for $\alpha\le10^{-2}$.  This is what Theorem~\ref{thm:cluster} predicts
once one knows that $\theta_{\min}$ is insensitive to $h$.

\begin{table}[htbp]
\centering
\caption{Iteration counts (IT) and total CPU times in seconds (setup plus solve) of GMRES$(20)$ preconditioned by $\widehat{\mathcal{Q}}_\alpha$ for different values of the regularization parameter $\alpha$ (Example~2, $\nu=\kappa=1$).}
\label{tab:alpha}
\footnotesize
\setlength{\tabcolsep}{4pt}
\begin{tabular}{rrrrrrrrrrr}
\toprule
 & \multicolumn{2}{c}{$\alpha=10^{-3}$} & \multicolumn{2}{c}{$\alpha=10^{-2}$} & \multicolumn{2}{c}{$\alpha=10^{-1}$} & \multicolumn{2}{c}{$\alpha=1$} & \multicolumn{2}{c}{$\alpha=10^{1}$} \\
\cmidrule(lr){2-3}\cmidrule(lr){4-5}\cmidrule(lr){6-7}\cmidrule(lr){8-9}\cmidrule(lr){10-11}
$n$ & IT & CPU & IT & CPU & IT & CPU & IT & CPU & IT & CPU \\
\midrule
16 & 6 & 0.002 & 8 & 0.001 & 12 & 0.002 & 20 & 0.002 & 35 & 0.004 \\
32 & 6 & 0.006 & 9 & 0.006 & 13 & 0.007 & 24 & 0.008 & 39 & 0.011 \\
64 & 8 & 0.025 & 9 & 0.024 & 14 & 0.028 & 26 & 0.037 & 40 & 0.048 \\
128 & 8 & 0.117 & 9 & 0.119 & 14 & 0.140 & 27 & 0.188 & 40 & 0.234 \\
\bottomrule
\end{tabular}
\end{table}

In the left panel of Figure~\ref{fig:alphaspec} one can watch the third group
of eigenvalues migrate towards $1$ as $\alpha$ decreases, with the clusters at
$\pm1$ staying put.  The right panel and Table~\ref{tab:bound} compare
$\max_i|1-\mu_i|$ with the bound \eqref{eq:sharpbound} and with the cruder
bound \eqref{eq:clusterbound} of Theorem~\ref{thm:cluster}.  The sharper
bound overestimates by a factor of at most $1.4$; it is correct to two digits
at $\alpha=10^{-3}$ and to three from $\alpha=10^{-4}$ on, and it is
asymptotically exact as $\alpha\downarrow0$.

\begin{figure}[htbp]
\centering
\includegraphics[width=0.95\textwidth]{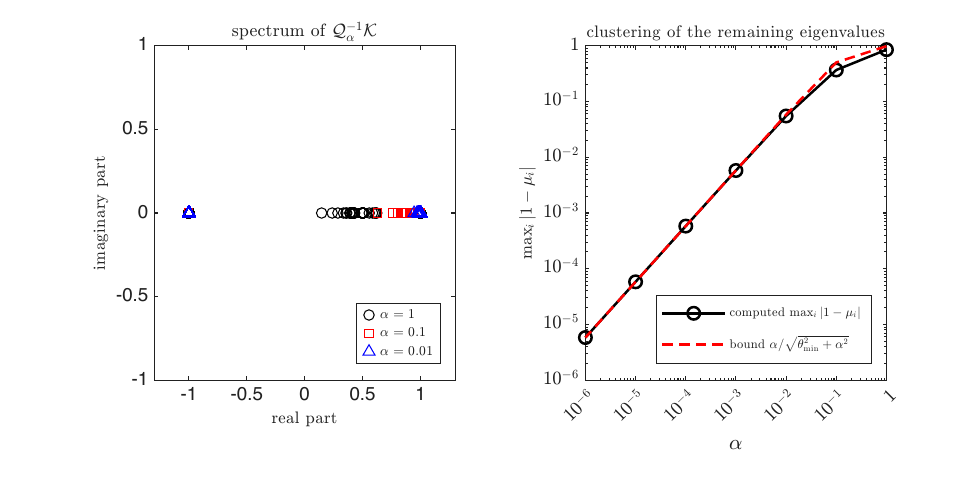}
\caption{Left: eigenvalues of $\Q_\alpha^{-1}\K$ for three values of $\alpha$
($n=8$, Example~\ref{ex:3}, $\nu=\kappa=1$, exact $\Sc_1$).  Right: computed
$\max_i|1-\mu_i|$ against the bound \eqref{eq:sharpbound} of
Proposition~\ref{prop:disk}.}
\label{fig:alphaspec}
\end{figure}

\begin{table}[htbp]
\centering
\caption{Sharpness of the eigenvalue bound \eqref{eq:sharpbound} of Proposition~\ref{prop:disk} for Example~2 with $n=8$, $\nu=\kappa=1$ and the exact Schur complement $\mathcal{S}_1$.}
\label{tab:bound}
\small
\setlength{\tabcolsep}{4pt}
\begin{tabular}{cccc}
\toprule
$\alpha$ & $\max_i|1-\mu_i|$ & $\alpha/\sqrt{\theta_{\min}^2+\alpha^2}$ & $\alpha/\theta_{\min}$ \\
\midrule
$1$ & $8.54\!\times\!10^{-1}$ & $9.86\!\times\!10^{-1}$ & $5.83\!\times\!10^{0}$ \\
$10^{-1}$ & $3.68\!\times\!10^{-1}$ & $5.04\!\times\!10^{-1}$ & $5.83\!\times\!10^{-1}$ \\
$10^{-2}$ & $5.51\!\times\!10^{-2}$ & $5.82\!\times\!10^{-2}$ & $5.83\!\times\!10^{-2}$ \\
$10^{-3}$ & $5.80\!\times\!10^{-3}$ & $5.83\!\times\!10^{-3}$ & $5.83\!\times\!10^{-3}$ \\
$10^{-4}$ & $5.83\!\times\!10^{-4}$ & $5.83\!\times\!10^{-4}$ & $5.83\!\times\!10^{-4}$ \\
$10^{-5}$ & $5.83\!\times\!10^{-5}$ & $5.83\!\times\!10^{-5}$ & $5.83\!\times\!10^{-5}$ \\
$10^{-6}$ & $5.83\!\times\!10^{-6}$ & $5.83\!\times\!10^{-6}$ & $5.83\!\times\!10^{-6}$ \\
\bottomrule
\end{tabular}
\end{table}

\subsection{Comparison of the preconditioners}

In this subsection we compare the four practical preconditioners
$\widehat\M_1$, $\widehat\M_2$, $\widehat\M_3$ and $\Qh_\alpha$ on
Example~\ref{ex:3} with $\nu=\kappa=1$, $\alpha=10^{-2}$ and drop tolerance
$10^{-2}$.  The iteration counts and the total CPU times of GMRES$(20)$ are
reported in Table~\ref{tab:compare}, and the corresponding relative
preconditioned residual norms for $n=64$ are displayed in
Figure~\ref{fig:residuals}.

The following observations can be made from Table~\ref{tab:compare}.

\begin{enumerate}\itemsep2pt
\item The iteration counts of $\Qh_\alpha$ are essentially independent of the
      mesh size: they equal $8$ for $n=16$ and $9$ for all $n$ from $32$ to
      $1024$, the largest system having more than $4\times10^{6}$ unknowns.
\item The iteration counts of $\widehat\M_3$ increase monotonically from $16$
      to $22$ as $n$ increases from $16$ to $1024$, which is consistent with
      the results reported in \cite{GreifHe2023}.  Those of $\widehat\M_1$
      and $\widehat\M_2$ remain between $54$ and $60$ and between $52$ and
      $64$ respectively, and exhibit no mesh-independent behavior.
\item Consequently, $\Qh_\alpha$ requires between $2.0$ and $2.4$ times fewer
      iterations than $\widehat\M_3$ and between $6.2$ and $6.8$ times fewer
      than $\widehat\M_1$, and the gap widens as the mesh is refined.
\item The total CPU time of $\Qh_\alpha$ is the smallest of the four for
      every $n$, and the reduction relative to $\widehat\M_1$ and
      $\widehat\M_2$ reaches factors of $2.5$ and $2.4$ respectively.
\end{enumerate}

Figure~\ref{fig:residuals} confirms that these counts do not depend on the
particular stopping criterion: the residual curve of $\Qh_\alpha$ decreases
more rapidly than those of the other three preconditioners throughout the
iteration.

\begin{table}[htbp]
\centering
\caption{Iteration counts (IT) and CPU times in seconds of GMRES$(20)$ for the four practical preconditioners (Example~2, $\nu=\kappa=1$, $\alpha=10^{-2}$, drop tolerance $10^{-2}$).}
\label{tab:compare}
\footnotesize
\setlength{\tabcolsep}{4pt}
\begin{tabular}{rrrrrrrrr}
\toprule
 & \multicolumn{2}{c}{$\widehat{\mathcal{M}}_1$} & \multicolumn{2}{c}{$\widehat{\mathcal{M}}_2$} & \multicolumn{2}{c}{$\widehat{\mathcal{M}}_3$} & \multicolumn{2}{c}{$\widehat{\mathcal{Q}}_\alpha$} \\
\cmidrule(lr){2-3}\cmidrule(lr){4-5}\cmidrule(lr){6-7}\cmidrule(lr){8-9}
$n$ & IT & CPU & IT & CPU & IT & CPU & IT & CPU \\
\midrule
16 & 54 & 0.004 & 52 & 0.003 & 16 & 0.002 & 8 & 0.002 \\
32 & 56 & 0.012 & 52 & 0.011 & 18 & 0.006 & 9 & 0.005 \\
64 & 60 & 0.047 & 56 & 0.045 & 18 & 0.023 & 9 & 0.019 \\
128 & 58 & 0.198 & 59 & 0.206 & 19 & 0.105 & 9 & 0.091 \\
256 & 58 & 0.930 & 62 & 0.973 & 20 & 0.500 & 9 & 0.471 \\
512 & 56 & 4.533 & 64 & 4.916 & 21 & 2.794 & 9 & 2.543 \\
1024 & 56 & 23.019 & 60 & 23.498 & 22 & 15.574 & 9 & 15.330 \\
\bottomrule
\end{tabular}
\end{table}
\begin{table}[htbp]
\centering
\caption{Breakdown of the CPU times of Table~\ref{tab:compare} into the setup phase (incomplete Cholesky factorization and sparse factorization of the diagonal blocks of the preconditioner) and the GMRES iteration phase, in seconds.}
\label{tab:timing}
\footnotesize
\setlength{\tabcolsep}{4pt}
\begin{tabular}{rrrrrrrrr}
\toprule
 & \multicolumn{2}{c}{$\widehat{\mathcal{M}}_1$} & \multicolumn{2}{c}{$\widehat{\mathcal{M}}_2$} & \multicolumn{2}{c}{$\widehat{\mathcal{M}}_3$} & \multicolumn{2}{c}{$\widehat{\mathcal{Q}}_\alpha$} \\
\cmidrule(lr){2-3}\cmidrule(lr){4-5}\cmidrule(lr){6-7}\cmidrule(lr){8-9}
$n$ & setup & solve & setup & solve & setup & solve & setup & solve \\
\midrule
16 & 0.001 & 0.003 & 0.001 & 0.003 & 0.001 & 0.001 & 0.001 & 0.001 \\
32 & 0.003 & 0.009 & 0.003 & 0.008 & 0.003 & 0.003 & 0.004 & 0.002 \\
64 & 0.012 & 0.035 & 0.012 & 0.033 & 0.011 & 0.011 & 0.011 & 0.008 \\
128 & 0.051 & 0.147 & 0.052 & 0.154 & 0.052 & 0.053 & 0.050 & 0.041 \\
256 & 0.268 & 0.662 & 0.260 & 0.713 & 0.262 & 0.239 & 0.271 & 0.200 \\
512 & 1.666 & 2.867 & 1.613 & 3.303 & 1.598 & 1.196 & 1.643 & 0.891 \\
1024 & 10.775 & 12.243 & 10.342 & 13.156 & 10.191 & 5.383 & 10.927 & 4.403 \\
\bottomrule
\end{tabular}
\end{table}

\begin{figure}[htbp]
\centering
\includegraphics[width=0.7\textwidth]{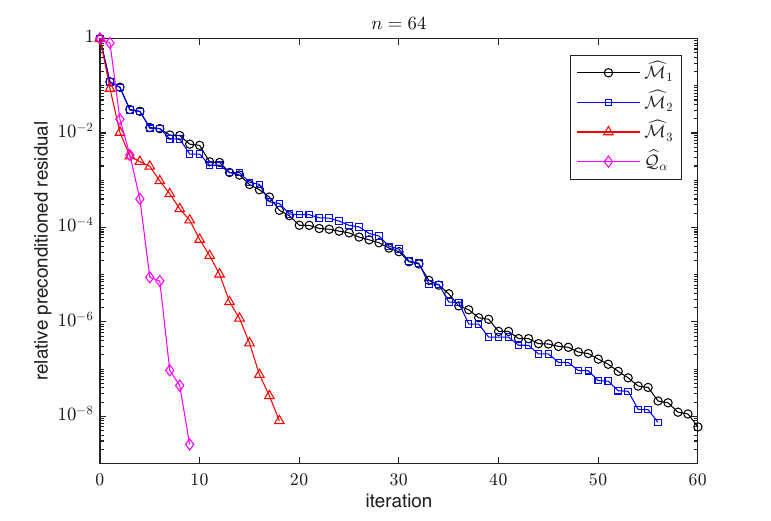}
\caption{Relative preconditioned residual norms for the four practical
preconditioners, Example~\ref{ex:3} with $n=64$ and $\nu=\kappa=1$.  The
stopping tolerance is $10^{-8}$, as everywhere else, so the curves end at the
iteration counts of the $n=64$ row of Table~\ref{tab:compare}.}
\label{fig:residuals}
\end{figure}

In order to examine the cost of the four preconditioners in more detail, the
CPU times of Table~\ref{tab:compare} are decomposed into the setup phase and
the GMRES iteration phase in Table~\ref{tab:timing}.  We observe the
following.

\begin{enumerate}\itemsep2pt
\item The four setup times are nearly identical: for $n\ge64$ they differ by
      at most $4$--$7\%$, and the setup time of $\Qh_\alpha$ exceeds the
      smallest of the four by at most $7\%$.  This is noteworthy, since all
      four preconditioners compute the same incomplete Cholesky factor of
      $A_d$ and the same sparse factorization of $A_d$, whereas
      $\widehat\M_1$, $\widehat\M_2$ and $\widehat\M_3$ subsequently
      factorize $\Sch_1$ and $\Qh_\alpha$ factorizes
      $W_\alpha=\Sch_1+\alpha^{-1}B\tr B$.  The grad--div term
      $\alpha^{-1}B\tr B$ couples the two velocity components over the whole
      domain, whereas $\Sch_1$ is nearly block diagonal with respect to them,
      so that $W_\alpha$ possesses considerably more fill-in.  Since the
      measured setup times nevertheless agree, this additional fill-in is not
      the dominant cost of the setup phase at the problem sizes considered
      here.
\item In the iteration phase $\Qh_\alpha$ is the cheapest of the four for
      every $n$, by factors of $1.2$--$1.5$ with respect to $\widehat\M_3$
      and of $2.8$--$4.5$ with respect to $\widehat\M_1$.
\item The total times are therefore determined by the iteration phase.  For a
      single right-hand side, $\Qh_\alpha$ is faster than $\widehat\M_3$ for
      all $n\ge32$, the reduction being $17\%$, $17\%$, $13\%$, $6\%$, $9\%$
      and $2\%$ for $n=32,64,128,256,512,1024$; the two are equal at $n=16$,
      where the times are at the resolution of the measurement.  At $n=256$
      the total times of $\widehat\M_1$ and $\widehat\M_2$ are $2.0$ and
      $2.1$ times that of $\Qh_\alpha$.
\item The advantage of $\Qh_\alpha$ over $\widehat\M_3$ in total time
      decreases on the finest meshes.  This is because the setup phase, which
      the two preconditioners share almost exactly, grows faster with $n$
      than the iteration phase; at $n=1024$ it accounts for $71\%$ of the
      total time of $\Qh_\alpha$.
\end{enumerate}

In summary, $\Qh_\alpha$ attains mesh-independent iteration counts at a setup
cost comparable with that of the existing block preconditioners, and it is
the fastest of the four in total time on every mesh considered.  Since the
setup is performed once whereas the iteration phase is repeated for each
right-hand side, the advantage of $\Qh_\alpha$ increases when the
factorization is reused, and it becomes decisive for small values of $\nu$
and $\kappa$ (see Section~\ref{sec:robust}), where $\widehat\M_3$ requires
several times as many iterations or fails to converge.  A further reduction
of the setup cost, for instance by replacing the direct solve for $W_\alpha$
with an inexact one, is left for future work.

\subsection{Quality of the Schur complement approximation}

Table~\ref{tab:inexact} isolates the effect of approximating the interface
block $T$, which is what Corollary~\ref{cor:pert} points to.  With the exact
$T$ the iteration count barely depends on $\kappa$.  With an incomplete
Cholesky factorization it degrades gracefully as the drop tolerance is
loosened, and the degradation is worse for small $\kappa$, where
$\|T\|=\mathcal{O}(\kappa^{-1})$ is large and so is $\|E\|$.  The
scaled-identity approximation $T\approx(3\kappa)^{-1}I$ of
\cite[Eq.~(4.10)]{GreifHe2023} is adequate for moderate parameters but is not
competitive once $\kappa\le10^{-4}$, and on the $n=128$ grid it fails outright
at $\kappa=10^{-6}$.  The interface block is the critical ingredient.

\begin{table}[htbp]
\centering
\caption{Iteration counts of GMRES$(20)$ preconditioned by $\widehat{\mathcal{Q}}_\alpha$ with different approximations of the interface block $T$ of \eqref{eq:GAG}, for $\nu=1$, $\alpha=10^{-2}$ and Example~2; the rows marked ``exact'' use $\mathcal{Q}_\alpha$ itself. ``IC$(\varepsilon)$'' denotes an incomplete Cholesky factorization with drop tolerance $\varepsilon$.}
\label{tab:inexact}
\small
\setlength{\tabcolsep}{4pt}
\begin{tabular}{lrrrrr}
\toprule
& & \multicolumn{4}{c}{$\kappa$} \\
\cmidrule(lr){3-6}
approximation of $T$ & $n$ & $1$ & $10^{-2}$ & $10^{-4}$ & $10^{-6}$ \\
\midrule
exact & 32 & 6 & 6 & 7 & 8 \\
exact & 64 & 6 & 6 & 7 & 8 \\
IC($10^{-3}$) & 32 & 8 & 9 & 13 & 16 \\
IC($10^{-3}$) & 64 & 8 & 11 & 15 & 17 \\
IC($10^{-3}$) & 128 & 9 & 12 & 18 & 21 \\
IC($10^{-2}$) & 32 & 9 & 12 & 16 & 19 \\
IC($10^{-2}$) & 64 & 9 & 12 & 20 & 35 \\
IC($10^{-2}$) & 128 & 9 & 13 & 34 & 38 \\
IC($10^{-1}$) & 32 & 9 & 13 & 19 & 24 \\
IC($10^{-1}$) & 64 & 9 & 13 & 28 & 39 \\
IC($10^{-1}$) & 128 & 9 & 14 & 40 & 48 \\
\bottomrule
\end{tabular}
\end{table}

\subsection{Robustness with respect to the physical parameters}
\label{sec:robust}

Tables~\ref{tab:robust-nu1} and \ref{tab:robust-nu2} report iteration counts
as $\kappa$ ranges over eight orders of magnitude, at $\nu=1$ and
$\nu=10^{-2}$.  Three methods are compared: $\widehat\M_3$, $\Qh_\alpha$ with
the fixed value $\alpha=10^{-2}$, and $\Qh_\alpha$ with the scaled choice
\eqref{eq:alphascaled}.

In the well-scaled regime $\kappa\ge10^{-2}$ the proposed preconditioner is
better everywhere, needing $8$--$19$ iterations against $12$--$36$ for
$\widehat\M_3$, at every mesh size and for both values of $\nu$.  Over the
whole set of $72$ cases, $\Qh_\alpha$ converged in all of them and
$\widehat\M_3$ stagnated in seven, all at small $\kappa$ on the finer grids,
where the diagonal surrogate $\Sch_2$ no longer represents the second Schur
complement.  That is the practical payoff of doing without $\Sch_2$.

The scaled choice \eqref{eq:alphascaled} matters a great deal once
$\kappa$ is small.  With $\alpha$ fixed at $10^{-2}$ the count reaches $102$ at
$n=256$, $\nu=1$, $\kappa=10^{-6}$; scaling $\alpha$ brings it down to $27$
and keeps every entry of both tables at or below $34$, the largest values
occurring at $\nu=10^{-2}$, $n=256$, $\kappa\le10^{-7}$, where
\eqref{eq:alphascaled} gives $\alpha\le10^{-9}$ and the conditioning of
$W_\alpha$ begins to tell.  Again this follows the
theory: by Theorem~\ref{thm:cluster} the clustering is controlled by
$\alpha/\theta_{\min}$, and Table~\ref{tab:assumption} shows $\theta_{\min}$
shrinking in proportion to $\kappa$.

One effect runs the other way and should be recorded.  For $\nu=10^{-2}$ and
$n=256$, $\widehat\M_3$ becomes markedly \emph{faster} as $\kappa$ decreases,
converging in $14$ iterations for every $\kappa\le10^{-4}$ and beating
$\Qh_\alpha$ there.  Non-monotone behaviour of $\widehat\M_3$ on the finest
grids was also reported in \cite[Tab.~9]{GreifHe2023}.  
It does
not survive the change to $\nu=1$, where in the same corner $\widehat\M_3$
needs $264$ iterations at $\kappa=10^{-5}$ and fails outright for
$\kappa\le10^{-6}$.  Figure~\ref{fig:robust} summarizes the
comparison on the finest grid.

\begin{table}[htbp]
\centering
\caption{Iteration counts of GMRES$(20)$ for $\nu=1$ and varying $n$ and $\kappa$ (Example~2). $\widehat{\mathcal{Q}}_\alpha$ is used both with the fixed parameter $\alpha=10^{-2}$ and with the scaled choice $\alpha=10^{-2}\min\{1,\kappa\}$ of Section~\ref{sec:alphachoice}. The symbol $-$ marks no convergence within $500$ iterations.}
\label{tab:robust-nu1}
\footnotesize
\setlength{\tabcolsep}{4pt}
\begin{tabular}{lrrrrrrrrrr}
\toprule
& & \multicolumn{9}{c}{$\kappa$} \\
\cmidrule(lr){3-11}
preconditioner & $n$ & $1$ & $10^{-1}$ & $10^{-2}$ & $10^{-3}$ & $10^{-4}$ & $10^{-5}$ & $10^{-6}$ & $10^{-7}$ & $10^{-8}$ \\
\midrule
$\widehat{\mathcal{M}}_3$ & 32 & 18 & 18 & 19 & 19 & 20 & 20 & 24 & 37 & 44 \\
$\widehat{\mathcal{Q}}_\alpha$, $\alpha=10^{-2}$ & 32 & 9 & 9 & 12 & 14 & 16 & 17 & 19 & 19 & 22 \\
$\widehat{\mathcal{Q}}_\alpha$, scaled $\alpha$ & 32 & 9 & 8 & 8 & 10 & 11 & 10 & 12 & 12 & 13 \\
$\widehat{\mathcal{M}}_3$ & 64 & 18 & 19 & 22 & 26 & 36 & 40 & 46 & 47 & 77 \\
$\widehat{\mathcal{Q}}_\alpha$, $\alpha=10^{-2}$ & 64 & 9 & 10 & 12 & 16 & 20 & 24 & 35 & 38 & 39 \\
$\widehat{\mathcal{Q}}_\alpha$, scaled $\alpha$ & 64 & 9 & 8 & 9 & 12 & 13 & 14 & 15 & 16 & 19 \\
$\widehat{\mathcal{M}}_3$ & 128 & 19 & 20 & 23 & 37 & 40 & 43 & 45 & 48 & $-$ \\
$\widehat{\mathcal{Q}}_\alpha$, $\alpha=10^{-2}$ & 128 & 9 & 10 & 13 & 19 & 34 & 44 & 38 & 40 & 45 \\
$\widehat{\mathcal{Q}}_\alpha$, scaled $\alpha$ & 128 & 9 & 9 & 10 & 13 & 16 & 18 & 18 & 21 & 23 \\
$\widehat{\mathcal{M}}_3$ & 256 & 20 & 21 & 25 & 40 & 44 & 264 & $-$ & $-$ & $-$ \\
$\widehat{\mathcal{Q}}_\alpha$, $\alpha=10^{-2}$ & 256 & 9 & 10 & 14 & 20 & 52 & 50 & 102 & 97 & 102 \\
$\widehat{\mathcal{Q}}_\alpha$, scaled $\alpha$ & 256 & 9 & 9 & 10 & 14 & 19 & 26 & 27 & 28 & 29 \\
\bottomrule
\end{tabular}
\end{table}
\begin{table}[htbp]
\centering
\caption{Iteration counts of GMRES$(20)$ for $\nu=10^{-2}$ and varying $n$ and $\kappa$ (Example~2); notation as in Table~\ref{tab:robust-nu1}.}
\label{tab:robust-nu2}
\footnotesize
\setlength{\tabcolsep}{4pt}
\begin{tabular}{lrrrrrrrrrr}
\toprule
& & \multicolumn{9}{c}{$\kappa$} \\
\cmidrule(lr){3-11}
preconditioner & $n$ & $1$ & $10^{-1}$ & $10^{-2}$ & $10^{-3}$ & $10^{-4}$ & $10^{-5}$ & $10^{-6}$ & $10^{-7}$ & $10^{-8}$ \\
\midrule
$\widehat{\mathcal{M}}_3$ & 32 & 15 & 17 & 12 & 19 & 21 & 39 & 40 & 40 & 43 \\
$\widehat{\mathcal{Q}}_\alpha$, $\alpha=10^{-2}$ & 32 & 8 & 9 & 11 & 12 & 12 & 13 & 14 & 15 & 16 \\
$\widehat{\mathcal{Q}}_\alpha$, scaled $\alpha$ & 32 & 8 & 9 & 10 & 9 & 9 & 10 & 10 & 12 & 12 \\
$\widehat{\mathcal{M}}_3$ & 64 & 15 & 19 & 23 & 36 & 37 & 37 & 39 & 41 & $-$ \\
$\widehat{\mathcal{Q}}_\alpha$, $\alpha=10^{-2}$ & 64 & 9 & 11 & 13 & 14 & 15 & 16 & 18 & 19 & 20 \\
$\widehat{\mathcal{Q}}_\alpha$, scaled $\alpha$ & 64 & 9 & 11 & 12 & 12 & 13 & 13 & 14 & 16 & 17 \\
$\widehat{\mathcal{M}}_3$ & 128 & 15 & 20 & 36 & 14 & 34 & 38 & 40 & $-$ & $-$ \\
$\widehat{\mathcal{Q}}_\alpha$, $\alpha=10^{-2}$ & 128 & 10 & 13 & 17 & 18 & 19 & 23 & 26 & 39 & 59 \\
$\widehat{\mathcal{Q}}_\alpha$, scaled $\alpha$ & 128 & 10 & 11 & 15 & 16 & 16 & 17 & 20 & 21 & 24 \\
$\widehat{\mathcal{M}}_3$ & 256 & 15 & 20 & 36 & 16 & 14 & 14 & 14 & 14 & 14 \\
$\widehat{\mathcal{Q}}_\alpha$, $\alpha=10^{-2}$ & 256 & 10 & 14 & 19 & 29 & 37 & 42 & 44 & 44 & 51 \\
$\widehat{\mathcal{Q}}_\alpha$, scaled $\alpha$ & 256 & 10 & 14 & 18 & 20 & 21 & 25 & 28 & 30 & 34 \\
\bottomrule
\end{tabular}
\end{table}

\begin{figure}[htbp]
\centering
\includegraphics[width=0.95\textwidth]{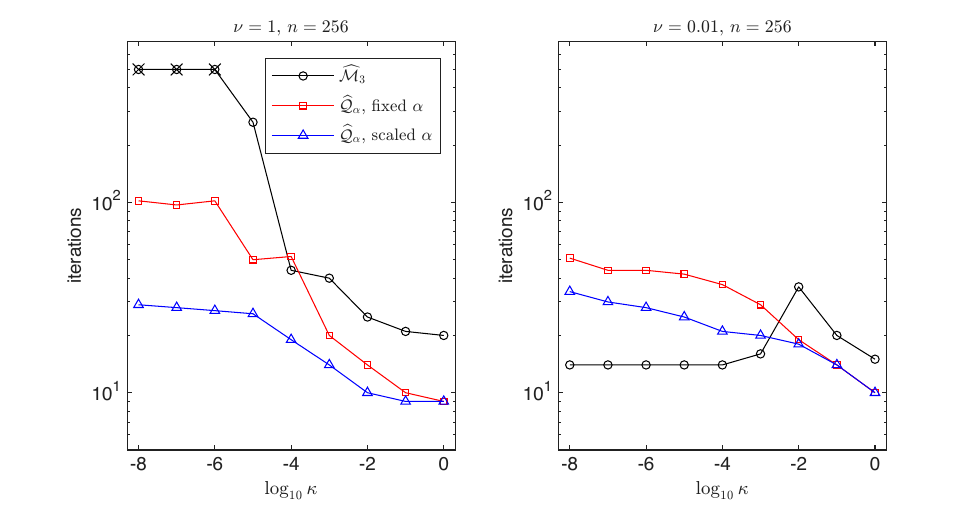}
\caption{Iteration counts as a function of $\kappa$ for $\widehat\M_3$ and for
$\Qh_\alpha$ with fixed and with scaled $\alpha$ (Example~\ref{ex:3}).
Crosses mark runs that did not converge within $500$ iterations.}
\label{fig:robust}
\end{figure}

\subsection{A computed solution}

Figure~\ref{fig:solution} shows the solution of Example~\ref{ex:1} on a
$64\times128$ grid.  Panels (a) and (b) give the computed Darcy potential and
Stokes pressure, which are plotted separately because the two fields differ by
an order of magnitude; the jump across $\Gamma$ allowed by \eqref{eq:if-normal}
is clearly resolved.  Panel (c) shows the velocity on the whole domain, the
Stokes velocity $(u_h,v_h)$ above $\Gamma$ and the Darcy velocity
$-\kappa\nabla\phi_h$ below it, with arrows drawn on top of the magnitude.  The
normal component passes through the interface without a visible kink, as
\eqref{eq:if-mass} requires, while the tangential component slips.  Panel (d)
plots the pointwise error of the combined potential/pressure field.  The error
is largest in the Stokes region, consistent with the error norms of
Table~\ref{tab:converr-ex1}, and it shows no layer along $\Gamma$.

\begin{figure}[htbp]
\centering
\includegraphics[width=0.9\textwidth]{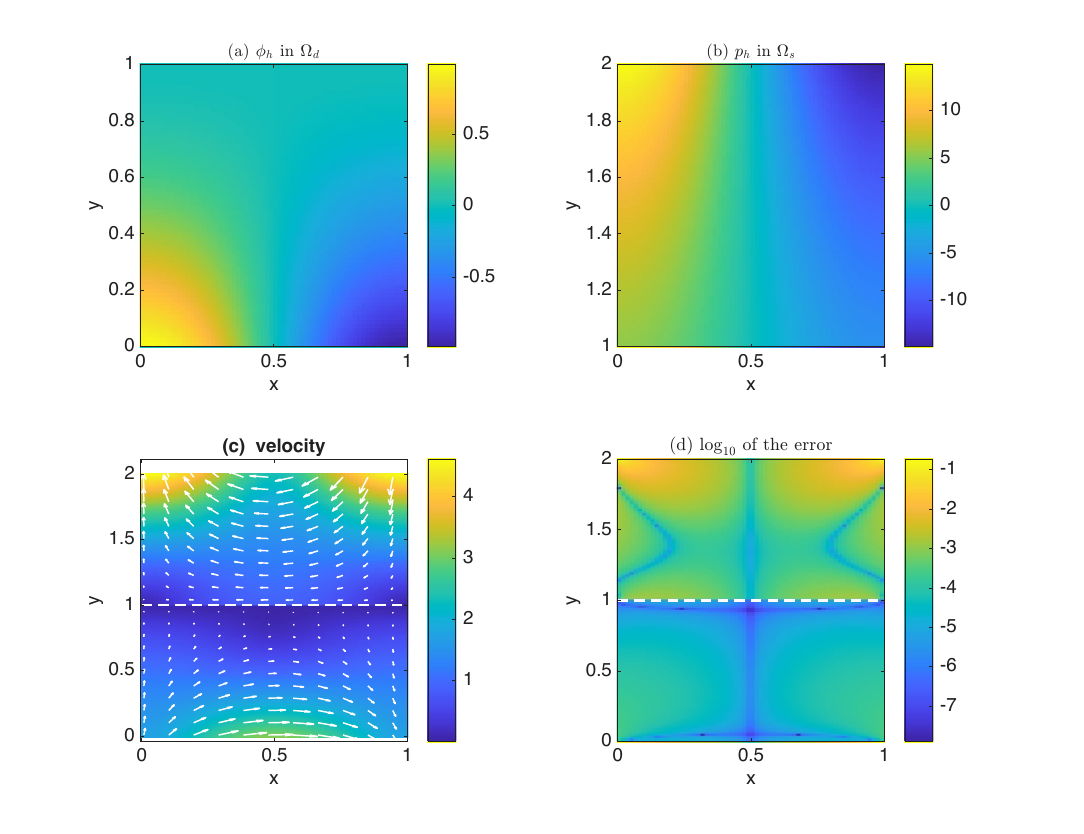}
\caption{Example~\ref{ex:1} with $n=64$: (a) computed Darcy potential
$\phi_h$; (b) computed Stokes pressure $p_h$; (c) velocity field, Stokes
velocity in $\Oms$ and Darcy velocity $-\kappa\nabla\phi_h$ in $\Omd$, arrows
over the magnitude; (d) base-ten logarithm of the pointwise error of the
combined field.  Dashed lines mark the interface $\Gamma$.}
\label{fig:solution}
\end{figure}

\FloatBarrier
\section{Concluding remarks}
\label{sec:conclusions}

We have proposed and analysed a coupled block preconditioner $\Q_\alpha$ for
the double saddle-point systems produced by the MAC discretization of the
Stokes--Darcy equations.  What distinguishes it from the block triangular
preconditioners in the literature is that the second Schur complement is
replaced by a scalar regularization instead of being approximated, so that the
only object one has to approximate is the first Schur complement, whose
structure is explicitly known.

The spectral analysis is complete.  The preconditioned matrix has $1$ and $-1$
as eigenvalues with multiplicities $n^2$ and $2n^2-n$, and the remaining $n^2$
eigenvalues are the M\"obius images $\theta/(\theta+\alpha)$ of the spectrum
of the second Schur complement.  From this we obtain the clustering rate
$\mathcal{O}(\alpha/\theta_{\min})$, a localization inside the disk
$|z-\tfrac12|<\tfrac12$, and a quantitative rule for scaling the
regularization parameter with the physical parameters.  For the inexact
variant, every eigenvalue other than $0$ and $1$ solves a scalar quadratic
with Rayleigh-quotient coefficients, which yields an
$\mathcal{O}(\|E\|^{1/2})$ bound on the distance of the spectrum from its
ideal location, with $E$ confined to the $n$ interface unknowns.

Numerically, at unit parameters the preconditioner cuts the GMRES iteration
count by a factor between two and seven relative to the block triangular
preconditioners and shows no dependence on the mesh size.  With the scaled
parameter \eqref{eq:alphascaled} it converges in every one of the $72$
configurations of the parameter study, across eight orders of magnitude of
the permeability, including the seven in which $\widehat\M_3$ stagnates.
The exception is $\nu=10^{-2}$ on the finest grid, where $\widehat\M_3$, in
the cases for which it converges at all, needs the fewer iterations.  Its
main cost is the setup, since $W_\alpha=\Sch_1+\alpha^{-1}B\tr B$ is denser
than $\Sch_1$.

Several questions remain open.  The most pressing is whether the direct solve
for $W_\alpha$ can be replaced by an inexact one, based on multigrid or on
augmented Lagrangian techniques; this would remove the setup bottleneck.
Theorem~\ref{thm:inexact} still applies, with $\Sch_1$ replaced by whatever
operator the inexact solve realizes, but $E$ is then no longer confined to the
interface and Lemma~\ref{lem:ranknE} is lost.  Extensions to three
dimensions, to the non-uniform grids of \cite{RuiSun2020}, and to the
time-dependent Navier--Stokes/Darcy coupling would all be worth pursuing.
Finally, the identity \eqref{eq:mobius} uses nothing specific to the
Stokes--Darcy problem, and we expect the analysis of
Section~\ref{sec:spectral} to carry over to other double saddle-point systems
with a nonzero $(2,2)$ block.

\section*{Author Contributions}

All authors contributed to the analysis and to the writing of the manuscript. Z.~Chao also designed the numerical test and performed the computations.

\section*{Acknowledgments}

Z. Chao was supported by an AMS-Simons Research Enhancement Grant for PUI Faculty.

\section*{Conflicts of Interest}

The authors declare no conflicts of interest.

\section*{Use of Artificial Intelligence}

Claude (Anthropic) was used to polish the English of the manuscript and check consistency. The
authors are responsible for the
scientific interpretation and for the final manuscript.

\section*{Data Availability Statement}

The MATLAB code that generates every table and figure in this paper is
available from the corresponding author upon reasonable request.

\bibliographystyle{plainnat}
\bibliography{refs}

\end{document}